\documentclass[final,1p,times]{elsarticle}
\usepackage{amscd,bbm,amsmath,amsthm,graphicx,amssymb,mathrsfs,latexsym,dsfont,amsfonts,paralist,epsfig,graphics,exscale}
\usepackage[colorlinks=true]{hyperref}
\usepackage{color,soul}%%%\hl{ highlighting text } % use hl{ } to highlight the sentence.

\newcommand{\bS}{\boldsymbol{S}}
\newtheorem*{thm-A}{Theorem A}
\newtheorem*{thm-B}{Theorem B}
\newtheorem*{thm-C}{Theorem C}
\newtheorem*{thm-D}{Theorem D}
\newtheorem*{corollary}{Corollary}
\newtheorem*{splitting}{Splitting Theorem}
\newtheorem*{conjecture}{Conjecture}
\newtheorem{theorem}{Theorem}[section]

\newtheorem{cor}[theorem]{Corollary}
\newtheorem{prop}[theorem]{Proposition}
\newtheorem{lem}[theorem]{Lemma}

\newtheorem{defn}[theorem]{Definition}
\theoremstyle{definition}

\newtheorem{example}[theorem]{\bf Example}
\newtheorem{remark}{{\it Remark}}

\numberwithin{equation}{section}
\newcommand{\EQ}{\begin{equation}\begin{array}{lllllllll}}
\newcommand{\EE}{\end{array}\end{equation}}
\newcommand{\MT}{\left[ \begin{array}{ccccccccc}}
\newcommand{\EM}{\end{array}\right]}

\newcommand{\eq}{\begin{equation}\begin{array}{lclllllllllllllll}}
\newcommand{\ee}{\end{array}\end{equation}}
\newcommand{\bmt}{\left[ \begin{array}{ccccccccc}}
\newcommand{\emt}{\end{array}\right]}
\newcommand{\bea}{\begin{eqnarray}}
\newcommand{\eea}{\end{eqnarray}}
\newcommand{\bean}{\begin{eqnarray*}}
\newcommand{\eean}{\end{eqnarray*}}

\journal{xxx}
\begin{document}

\begin{frontmatter}

%% Title, authors and addresses

%% use the tnoteref command within \title for footnotes;
%% use the tnotetext command for the associated footnote;
%% use the fnref command within \author or \address for footnotes;
%% use the fntext command for the associated footnote;
%% use the corref command within \author for corresponding author footnotes;
%% use the cortext command for the associated footnote;
%% use the ead command for the email address,
%% and the form \ead[url] for the home page:
%%
\title{Stability Criteria via Common Non-strict Lyapunov Matrix for Discrete-time Linear Switched Systems}

\tnotetext[label1]{Project was supported partly by National Natural Science Foundation of
China (Grant Nos.~11071112 and 11071263), the NSF of Guangdong Province and in part by NSF 0605181 and 1021203 of the United States.}%
%% use optional labels to link authors explicitly to addresses:
 \author[label2]{Xiongping Dai}
 \address[label2]{Department of Mathematics, Nanjing University, Nanjing 210093, People's Republic of China}
 \ead{xpdai@nju.edu.cn}
%% \ead[url]{home page}
%% \fntext[label2]{}
%% \cortext[cor1]{}
%% \address{Address\fnref{label3}}
%% \fntext[label3]{}

 \author[label3]{Yu Huang}
 \address[label3]{Department of Mathematics, Zhongshan (Sun Yat-Sen) University, Guangzhou 510275, People's Republic of China}
 \ead{stshyu@mail.sysu.edu.cn}

 \author[label4]{Mingqing Xiao}
 \address[label4]{Department of Mathematics, Southern Illinois University,
Carbondale, IL 62901-4408, USA}
 \ead{mxiao@math.siu.edu}
%%%%%%%%%%%%%%%%%%%%%%%%%%%%%%%%%%%%%%

\begin{abstract}
Let $\bS=\{S_1,S_2\}\subset\mathbb{R}^{d\times d}$ have a common, but not necessarily strict, Lyapunov matrix (i.e. there exists a symmetric positive-definite matrix $P$ such that $P-S_k^TPS_k\ge0$ for $k=1,2$). Based on a splitting theorem of the state space $\mathbb{R}^d$ (Dai, Huang and Xiao, arXiv:1107.0132v1[math.PR]), we establish several stability criteria for the discrete-time linear switched dynamics
\begin{equation*}
x_n=S_{\sigma_{\!n}}\dotsm S_{\sigma_{\!1}}(x_0),\quad x_0\in\mathbb{R}^d\textrm{ and }n\ge1
\end{equation*}
governed by the switching signal $\sigma\colon\mathbb{N}\rightarrow\{1,2\}$. More specifically, let $\rho(A)$ stand for the spectral radius of a matrix $A\in\mathbb{R}^{d\times d}$, then the outline of results obtained in this paper are: (1) For the case $d=2$, $\bS$ is absolutely stable (i.e., $\|S_{\sigma_{\!n}}\dotsm S_{\sigma_{\!1}}\|\to0$ driven by all switching signals $\sigma$) if and only if $\rho(S_1),\rho(S_2)$ and $\rho(S_1S_2)$ all are less than $1$; (2) For the case $d=3$, $\bS$ is absolutely stable if and only if $\rho(A)<1\;\forall A\in\{S_1,S_2\}^\ell$ for $\ell=1,2,3,4,5,6$, and $8$. This further implies that for any $\bS=\{S_1,S_2\}\subset\mathbb{R}^{d\times d}$ with the generalized spectral radius $\rho(\bS)=1$ where $d=2$ or $3$, if $\bS$ has a common, but not strict in general, Lyapunov matrix, then $\bS$ possesses the spectral finiteness property.
\end{abstract}
%%%%%%%%%%%%%%%%%%%%%%%%%%%
\end{frontmatter}
%%%%%%%%%%%%%%%%%%%%%%%%%%%%%%%%%%%%%%%%%%%%%%%%%%%%%%%%%%%%%%%%
%%%%%%%%%%%%%%%%%%%%%%%%%%%%%%%%%%%%%%%%%%%%%%%%%%%%%%%%%%%%%%%%
\section{Introduction}\label{secI}%

\subsection{Motivations}%%%
Let $\mathbb{R}^{d\times d}$ be the standard topological space of all $d$-by-$d$ real matrices
where $2\le d<+\infty$, and for any $A\in\mathbb{R}^{d\times d}$, by $\rho(A)$ we denote the spectral radius of $A$. In addition, we identify $A$ with its induced operator $A(\cdot)\colon x\mapsto Ax$ for $x\in\mathbb{R}^d$.
Let $\bS=\{S_1,\dotsc, S_K\}\subset\mathbb{R}^{d\times d}$ be a finite set with $2\le K<+\infty$.  We consider the stability and stabilization of the
linear inclusion/control dynamics
\begin{equation}\label{eq1.1}
x_n\in\left\{S_1,\dotsc,S_K\right\}(x_{n-1}),\quad x_0\in\mathbb{R}^d\textrm{ and }n\ge1.
\end{equation}
As in \cite{DHX-aut, Dai-JDE}, we denote by $\varSigma_{\!K}^+$
the set of all admissible control signals $\sigma\colon\mathbb{N}\rightarrow\{1,\dotsc,K\}$, equipped with the standard product topology. Here and in the sequel $\mathbb{N}=\{1,2,\dotsc\}$ and for any $\sigma\in\varSigma_{\!K}^+$ we will simply write $\sigma(n)=\sigma_{\!n}$ for all $n\ge1$.

For any input $(x_0,\sigma)$, where $x_0\in\mathbb{R}^d$ is an
initial state and
$\sigma=(\sigma_{\!n})_{n=1}^{+\infty}\in\varSigma_{\!K}^+$ a control
(switching) signal, there is a unique output
$\langle x_n(x_0,\sigma)\rangle_{n=1}^{+\infty}$, called an orbit of the system (\ref{eq1.1}), which
corresponds to the unique solution of the discrete-time linear switched dynamics
\begin{equation}\label{eq1.2}
x_n=S_{\sigma_{\!n}}\cdots S_{\sigma_{\!1}}(x_0),\quad x_0\in\mathbb{R}^d\textrm{ and }n\ge1
\end{equation}
driven/governed by the switching signal $\sigma$. Then as usual, $\bS$ is called (asymptotically) \emph{stable} driven by $\sigma$ if
\begin{equation*}
\lim_{n\to+\infty}\|S_{\sigma_{\!n}}\cdots S_{\sigma_{\!1}}(x_0)\|=0\; \forall x_0\in\mathbb{R}^d;\quad \textrm{or equivalently},\quad\|S_{\sigma_{\!n}}\cdots S_{\sigma_{\!1}}\|\to0\textrm{ as }n\to+\infty.
\end{equation*}
$\bS$ is said to be \emph{absolutely stable} if it is stable driven by all switching signals $\sigma\in\varSigma_{\!K}^+$; see, e.g., \cite{Gur-LAA95}. We note that the stability of $\bS$ is independent of the norm $\|\cdot\|$ used here.

It is a well-known fact that if each member $S_k$ of $\bS$ shares a common Lyapunov matrix; i.e., there exists a symmetric positive-definite matrix $Q\in\mathbb{R}^{d\times d}$ such that
\begin{equation*}
Q-S_k^TQS_k>0\quad (1\le k\le K),
\end{equation*}
then $\bS$ is absolutely stable. Here $^T$ stands for the transpose operator of matrices or vectors. An essentially weak condition is that each member $S_k$ of $\bS$ shares a common, ``but not necessarily strict,'' Lyapunov matrix; that is, there exists a symmetric positive-definite matrix $P$ such that
\begin{subequations}\label{eq1.3}
\begin{align}
&P-S_k^TPS_k\ge0,\quad 1\le k\le K.\label{eq1.3a}\\
\intertext{Here ``$A\ge0$'' means $x^TAx\ge0\,\forall x\in\mathbb{R}^d$. Associated to the weak Lyapunov matrix $P$ as in (\ref{eq1.3a}), we define the vector norm on $\mathbb{R}^d$ as}&\|x\|_P=\sqrt{x^TPx}\quad\forall x\in\mathbb{R}^d.\label{eq1.3b}
\end{align}
\end{subequations}
(We also write its induced operator/matrix norm on $\mathbb{R}^{d\times d}$ as $\|\cdot\|_P$.) Then, $\|S_k\|_P\le1$ for all $1\le k\le K$. Condition (\ref{eq1.3a}) is both practically important and academically challenging, for example, \cite{MS, BL00, Hart02, BL05, Sun08} for the continuous-time case and \cite{Gur-LAA95} for discrete case. Indeed, it is desirable in many practical issues and is closely related to periodic solutions and limit cycles, see, e.g., \cite{Bar88-SMJ, Bar93} and \cite[Proposition~18]{PW-LAA08}; in addition, if $S_k, 1\le k\le K$, are paracontractive (i.e., $x^TS_k^TS_kx\le x^Tx$ for all $x\in\mathbb{R}^d$, and ``$=$'' holds if and only if $S_k(x)=x$, see, e.g., \cite{SVR:IEEE}), then condition (\ref{eq1.3a}) holds.

In this paper, we will study the stability of $\bS$ that satisfies condition (\ref{eq1.3a}). Even under condition (\ref{eq1.3a}), the stability of every subsystems $S_k$ does not implies the absolute stability of $\bS$, as shown by Example~\ref{exaVI.6} constructed in Section~\ref{secVI}. So, our stability criteria\,---\,Theorems~A, B, C, and D\,---\, established in this paper, are nontrivial.

\subsection{Stability driven by nonchaotic switching signals}\label{secI.2}%%%
Under condition (\ref{eq1.3a}), in \cite{Bal} for the continuous-time case, Balde and Jouan provided a large class of switching signals for which a large class of switched systems are stable, by considering nonchaotic inputs and the geometry of $\omega$-limit sets of the matrix sequences $\langle S_{\sigma_{\!n}}\dotsm S_{\sigma_{\!1}}\rangle_{n=1}^{+\infty}$.

Recall from \cite[Definition~1]{Bal} that a switching signal $\sigma=(\sigma_{\!n})_{n=1}^{+\infty}\in\varSigma_{\!K}^+$ is said to be \emph{nonchaotic}, if to any sequence $\langle n_i\rangle_{i\ge1}\nearrow+\infty$ and any $m\ge1$ there corresponds some integer $\delta$ with $2\le\delta\le m+1$ such that $\forall \ell_0\ge1$, $\exists \ell\ge \ell_0$ so that $\sigma$ is constant restricted to some subinterval of $[n_\ell, n_\ell+m]$ of length greater than or equal to $\delta$. A switching signal $\sigma\in\varSigma_{\!K}^+$ is said to be \emph{generic}~\cite{Gur-LAA95} (or \emph{regular} in \cite{Bal}) if each alphabet in $\{1,\dotsc, K\}$ appears infinitely many times in the sequence $\sigma=(\sigma_{\!n})_{n=1}^{+\infty}$.

Then our first stability criterion obtained in this paper can be stated as follows:

\begin{thm-A}%%%
Let $\bS=\{S_1,\dotsc,S_K\}\subset\mathbb{R}^{d\times d}$ satisfy condition (\ref{eq1.3a}) with $\rho(S_k)<1$ for all $1\le k\le K$. Then
\begin{equation*}
\|S_{\sigma_{\!n}}\cdots S_{\sigma_{\!1}}\|\to0\quad \textrm{as }n\to+\infty
\end{equation*}
for any nonchaotic switching signal $\sigma=(\sigma_{\!n})_{n=1}^{+\infty}\in\varSigma_{\!K}^+$.
\end{thm-A}

We note that in Theorem~A, if $\sigma$ is additionally generic (regular), then this statement is a direct consequence of \cite[Theorem~3]{Bal}. However, without the genericity of $\sigma$, here we need to explore an essential property of a nonchaotic switching signal; see Lemma~\ref{lemII.1} below. In the case of $d=2$ and $K=2$, an ergodic version of Theorem~A will be stated in Corollary~\ref{corV.3} in Section~\ref{secV}.

As is shown by Example~\ref{exaVI.6} mentioned before, under the assumption of Theorem~A, one cannot expect the stability of $\bS$ driven by an arbitrary switching signal.
%%%%%%%%%%%%%%%%%%%%%%%%%%%%%%%%%%%%%%%%%%%%%%%%%%%%%%%%%%%%%%%%
\subsection{A splitting theorem driven by recurrent signals}\label{secI.3}%%%%
Next, we consider another type of switching signal\,---\,recurrent switching signal, which does not need to be nonchaotic and balanced and which seems more general from the viewpoint of ergodic theory. In fact, all recurrent switching signals form a set of total measure $1$.

Corresponding to a switching signal $\sigma=(\sigma_{\!n})_{n=1}^{+\infty}\in\varSigma_{\!K}^+$, for the system $\bS$ we define two important subspaces of the state space $\mathbb{R}^d$:
\begin{align*}
&E^s(\sigma)=\left\{x_0\in\mathbb{R}^d\colon\|S_{\sigma_{\!n}}\cdots S_{\sigma_{\!1}}(x_0)\|_P\to0\textrm{ as }n\to+\infty\right\}\\
\intertext{and}&E^c(\sigma)=\left\{x_0\in\mathbb{R}^d\colon \exists\, \langle n_i\rangle_{i=1}^{+\infty}\nearrow+\infty\textrm{ such that }\lim_{i\to+\infty}S_{\sigma_{\!n_i}}\cdots S_{\sigma_{\!1}}(x_0)=x_0\right\};
\end{align*}
called, respectively, the \emph{stable} and \emph{central manifolds} of $\bS$ driven by $\sigma$. Here $E^s(\sigma)$ and $E^c(\sigma)$ are indeed independent of the norm $\|\cdot\|_P$.

A switching signal $\sigma=(\sigma_{\!n})_{n=1}^{+\infty}\in\varSigma_{\!K}^+$ is called \emph{recurrent} under the classical one-sided Markov shift transformation, $\theta\colon\sigma(\cdot)\mapsto\sigma(\cdot+1)$, of $\varSigma_{\!K}^+$, if for any $\ell\ge1$ there exists some $m$ sufficiently large such that
\begin{equation*}
(\sigma_{\!1}, \dotsc,\sigma_{\!\ell})=(\sigma_{\!1+m}, \dotsc,\sigma_{\!\ell+m}).
\end{equation*}
We have then, for $\bS$, the following important splitting theorem of the state space $\mathbb{R}^d$ based on a recurrent switching signal:

\begin{splitting}[\cite{DHX-AIM}]%%%%
Let $\bS=\{S_1,\dotsc,S_K\}\subset\mathbb{R}^{d\times d}$ satisfy condition (\ref{eq1.3a}). Then, for any recurrent switching signal $\sigma\in\varSigma_{\!K}^+$ it holds
\begin{equation*}
\mathbb{R}^{d}=E^s(\sigma)\oplus E^c(\sigma)\quad \textrm{and} \quad S_{\sigma_{\!1}}(E^{s/c}(\sigma))=E^{s/c}(\sigma(\cdot+1)).
\end{equation*}
\end{splitting}

This theorem is a special case of a more general result \cite[Theorem~B$^{\prime\prime}$]{DHX-AIM}. So in this case, if the central manifold $E^c(\sigma)=\{0\}$ then $\bS$ is stable driven by the recurrent switching signal $\sigma$. This splitting is in fact unique under the Lyapunov norm $\|\cdot\|_P$.
%%%%%%%%%%%%%%%%%%%%%%%%%%%%%%%%%%%%%
\subsection{Almost sure stability}\label{secI.4}%%%
Under condition (\ref{eq1.3a}), let $\mathbb{K}_{\|\cdot\|_P}(S_k)=\{x\in\mathbb{R}^d\colon\|S_k(x)\|_P=\|x\|_P\}$ for $1\le k\le K$. We note that if $\|S_k\|_P<1$ then $\mathbb{K}_{\|\cdot\|_P}(S_k)=\{0\}$.

Next, using the above splitting theorem, we can obtain the following almost sure stability criterion:

\begin{thm-B}%%%
Let $\bS=\{S_1,S_2\}\subset\mathbb{R}^{d\times d}$ satisfy (\ref{eq1.3a}) and $\mathbb{K}_{\|\cdot\|_P}(S_1)\cap\mathbb{K}_{\|\cdot\|_P}(S_2)=\{0\}$, where $d=2$ or $3$. Then, if $\mathbb{P}$ is a non-atomic ergodic probability measure of the one-sided Markov shift transformation$\theta\colon\varSigma_{2}^+\rightarrow\varSigma_{2}^+$ defined by $\sigma(\cdot)\mapsto\sigma(\cdot+1)$, there holds
\begin{equation*}
\|S_{\sigma_{\!n}}\cdots S_{\sigma_{\!1}}\|_P\to0\quad \textrm{as }n\to+\infty
\end{equation*}
for $\mathbb{P}$-a.e. $\sigma\in\varSigma_{2}^+$.
\end{thm-B}

We consider a simple example. Let $\bS=\{S_1,S_2\}$ with $S_1=\mathrm{diag}(\frac{1}{2},\frac{1}{2})$ and $S_2=\mathrm{diag}(1,1)$. Then, $\mathbb{K}_{\|\cdot\|_2}(S_1)=\{0\}$ and $\mathbb{K}_{\|\cdot\|_2}(S_2)=\mathbb{R}^2$, where $\|\cdot\|_2$ stands for the usual Euclidean norm. So, $\mathbb{K}_{\|\cdot\|_P}(S_1)\cap\mathbb{K}_{\|\cdot\|_P}(S_2)=\{0\}$. Clearly, $\bS$ is not absolutely stable. This shows that under the situation of Theorem~B, it is necessary to consider the almost sure stability.
%%%%%%%%%%%%%%%%%%%%%%%%%%%%%%%%%%%%%%%%%%%%%%%%%%%%%%%%%%
\subsection{Absolute stability and finiteness property}\label{secI.5}%%%
For absolute stability, we can obtain the following two criteria Theorems~C and D, which show the stability is decidable in the cases of $d=2,3$ under condition (\ref{eq1.3a}).

\begin{thm-C}%%%
Let $\bS=\{S_1,S_2\}\subset\mathbb{R}^{2\times 2}$ satisfy condition (\ref{eq1.3a}). Then, $\bS$ is absolutely stable if and only if $\rho(A)<1$ for all $A\in\{S_1,S_2\}^\ell$ for $\ell=1,2$.
\end{thm-C}

\begin{thm-D}%%%
Let $\bS=\{S_1,S_2\}\subset\mathbb{R}^{3\times 3}$ satisfy condition (\ref{eq1.3a}). Then, $\bS$ is absolutely stable if and only if $\rho(A)<1$ for all $A\in\{S_1,S_2\}^\ell$ for $\ell=1,2,3,4,5,6$, and $8$.
\end{thm-D}

On the other hand, the accurate computation of the \emph{generalized spectral radius} of $\bS$, introduced by Daubechies and Lagarias in \cite{DL-LAA92} as
\begin{equation*}
\rho(\bS)=\lim_{n\to+\infty}\max_{\sigma\in\varSigma_{\!K}^+}\sqrt[n]{\rho(S_{\sigma_{\!n}}\dotsm S_{\sigma_{\!1}})}\quad\left(~=\sup_{n\ge1}\max_{\sigma\in\varSigma_{\!K}^+}\sqrt[n]{\rho(S_{\sigma_{\!n}}\dotsm S_{\sigma_{\!1}})}\right),
\end{equation*}
is very important for many subjects. If one can find a finite-length word $(w_1,\dotsc,w_n)\in\{1,\dotsc,K\}^n$ for some $n\ge1$, which realizes $\rho(\bS)$, i.e.,
\begin{equation*}
\rho(\bS)=\sqrt[n]{\rho(S_{w_n}\dotsm S_{w_1})},
\end{equation*}
then $\bS$ is said to have the \emph{spectral finiteness property}. A brief survey for some recent progresses regarding the finiteness property can be found in \cite[$\S$1.2]{Dai-Koz}.

Under condition (\ref{eq1.3a}), we have $\rho(\bS)\le1$. If $\rho(\bS)<1$ then $\bS$ is absolutely stable; see, e.g., \cite{Gur-LAA95}. If $\rho(\bS)=1$ then $\|\cdot\|_P$ is just an extremal norm for $\bS$ (see \cite{Bar,Wir02,Dai-JMAA} for more details). In \cite{Gur-LAA95}, Gurvits proved that if $\bS$ has a \emph{polytope}\footnote{A norm $\|\cdot\|$ on $\mathbb{R}^d$ is called a (real) \emph{polytope norm}, if the unit sphere $\mathbb{S}_{\|\cdot\|}=\left\{x\in\mathbb{R}^d\colon\|x\|=1\right\}$ is a polytope in $\mathbb{R}^d$; see, e.g., \cite{Gur-LAA95}.} extremal norm on $\mathbb{R}^d$, then it has the spectral finiteness property. However, the Lyapunov norm $\|\cdot\|_P$ defined as in (\ref{eq1.3b}) does not need to be a polytope norm, for example, $P=I_d$ the identity matrix which is associated with the usual Euclidean norm $\|\cdot\|_2$ on $\mathbb{R}^d$.

As a consequence of the statements of Theorems~C and D, we can easily obtain the following spectral finiteness result.

\begin{corollary}%%%
Let $\bS=\{S_1,S_2\}\subset\mathbb{R}^{d\times d}$ satisfy condition (\ref{eq1.3a}) with $\rho(\bS)=1$. Then the following two statements hold.
\begin{enumerate}
\item[(1)] For the case $d=2$, there follows $1=\max\left\{\rho(S_1),\rho(S_2), \sqrt{\rho(S_1S_2)}\right\}$.

\item[(2)] In the case $d=3$, there holds $1=\max\left\{\sqrt[n]{\rho(S_{w_n}\dotsm S_{w_1})}\,|\,w\in\{1,2\}^n, n=1,2,3,4,5,6,8\right\}$.
\end{enumerate}
\end{corollary}

\begin{proof}
Let $d=2$. Assume $\max\left\{\rho(S_1),\rho(S_2), \sqrt{\rho(S_1S_2)}\right\}<1$. Then Theorem~C implies that $\bS$ is absolutely stable and so $\rho(\bS)<1$, a contradiction. Similarly, we can prove the statement in the case $d=3$.
\end{proof}
It should be pointed out that if $\rho(\bS)<1$, then $\rho(\bS)$ does not need to be attained by these maximum values defined as in the above corollary.

%%%%%%%%%%%%%%%%%%%%%%%%%%%
\subsection{Outline}%%%
The paper is organized as follows. We shall prove Theorem~A in Section~\ref{secII}. In fact, we will prove a more general result (Theorem~\ref{thmII.3}) than Theorem~A there. Since the above Splitting Theorem is very important for the proofs of Theorems B, C, and D, we will give some notes on it in Section~\ref{secIII}. Then, Theorem~B will be proved in Section~\ref{secIV}. Section~\ref{secV} will be devoted to proving Theorems C and D. We will construct some examples in Section~\ref{secVI} to illustrate applications of our Theorems stated here. Finally, we will end this paper with some concluding remarks in Section~\ref{secVII}.

%%%%%%%%%%%%%%%%%%%%%%%%%%%%%%%%%%%%%%%%%%%%%%%%%%%%%%%%%%%%%%%%%%%%%%%%%%%
%%%%%%%%%%%%%%%%%%%%%%%%%%%%%%%%%%%%%%%%%%%%%%%%%%%%%%%%%%%%%%%%%%%%%%%%%%%
\section{Switched systems driven by nonchaotic switching signals}\label{secII}%%%
This section is devoted to proving Theorem~A stated in Section~\ref{secI.2} under the guise of a more general result.

For any integer $2\le K<+\infty$, we recall that
a switching signal $\sigma=(\sigma_{\!n})_{n=1}^{+\infty}\in\varSigma_{\!K}^+$ is called \emph{nonchaotic}, if to any sequence $\langle n_i\rangle_{i\ge1}\nearrow+\infty$ and any $m\ge1$ there corresponds some $\delta$ with $2\le\sigma\le m+1$ such that for all $\ell_0\ge1$, there exists an $\ell\ge \ell_0$ so that $\sigma$ is constant restricted to some subinterval of $[n_\ell, n_\ell+m]$ of length greater than or equal to $\delta$. Clearly, a constant switching signal $\sigma$ with $\sigma(n)\equiv k$ is nonchaotic.

Then from definition, we can obtain the following lemma, which discovers the essential property of a nonchaotic switching signal.

\begin{lem}\label{lemII.1}%%
Let $\sigma=(\sigma_{\!n})_{n=1}^{+\infty}\in\varSigma_{\!K}^+$ be a nonchaotic switching signal. Then, there exists some alphabet $k\in\{1,\dotsc,K\}$ such that for any $\ell\ge1$ and any $\ell^\prime\ge1$, there exists an $n_\ell\ge\ell^\prime$ so that $\sigma_{n_\ell+1}=\dotsm=\sigma_{n_\ell+\ell}=k$.
\end{lem}

\begin{proof}
First, we can choose a sequence $\langle n_i\rangle_{i\ge1}\nearrow+\infty$ and some $k\in\{1,\dotsc,K\}$, which are such that $n_{i+1}-n_i\nearrow+\infty$ and $\sigma_{n_i}=k$ for all $i\ge1$. Now from the definition of nonchaotic property with $m=1$, it follows that we can choose a subsequence of $\langle n_i\rangle_{i\ge1}$, still write, without loss of generality, as $\langle n_i\rangle_{i\ge1}$, such that $\sigma_{n_i}=\sigma_{n_i+1}=k$ for all $i\ge1$. Repeating this procedure for $\langle n_{i}+1\rangle_{i\ge1}$ proves the statement.
\end{proof}

Lemma~\ref{lemII.1} shows that the $\omega$-limit set of a nonchaotic switching signal contains at least one constant switching signal, under the sense of the classical Markov shift transformation.

The following fact is a simple consequence of the classical Gel'fand spectral formula, which will be refined in Section~\ref{secV} for the Lyapunov norm $\|\cdot\|_P$.

\begin{lem}\label{lemII.2}%%
For any $A\in\mathbb{R}^{d\times d}$ and any matrix norm $\|\cdot\|$ on $\mathbb{R}^{d\times d}$, if $\rho(A)<1$ then there is an integer $N\ge1$ such that $\|A^N\|<1$.
\end{lem}

For $\bS=\{S_1,\dotsc,S_K\}\subset\mathbb{R}^{d\times d}$, it is said to be \emph{product bounded}, if there is a universal constant $\beta\ge1$ such that
\begin{equation*}
\|S_{\sigma_{\!n}}\cdots S_{\sigma_{\!1}}\|\le\beta\quad\forall \sigma\in\varSigma_{\!K}^+\textrm{ and }n\ge1.
\end{equation*}
This property does not depend upon the norm $\|\cdot\|$ used here.

If $\bS$ is product bounded, then one always can choose a vector norm $\|\cdot\|$ on $\mathbb{R}^d$ such that its induced operator norm $\|\cdot\|$ on $\mathbb{R}^{d\times d}$ is such that $\|S_k\|\le1$ for all $1\le k\le K$. Then the norm $\|\cdot\|$ on $\mathbb{R}^d$ acts as a Lyapunov function for $\bS$. However, there does not need to exist a common, not strict in general, ``quadratic" Lyapunov function/matrix $P$ as in (\ref{eq1.3a}). So, the following theorem is more general than Theorem~A stated in Section~\ref{secI.2}.

\begin{theorem}\label{thmII.3}%%%
Let $\bS=\{S_1,\dotsc,S_K\}\subset\mathbb{R}^{d\times d}$ be product bounded. If $\rho(S_k)<1$ for all $1\le k\le K$, then $\bS$ is stable driven by any nonchaotic switching signals $\sigma\in\varSigma_{\!K}^+$.
\end{theorem}

\begin{proof}
Without loss of generality, let $\|\cdot\|$ be a matrix norm on $\mathbb{R}^{d\times d}$ such that $\|S_k\|\le1$ for all $1\le k\le K$. Let $\sigma=(\sigma_{\!n})_{n=1}^{+\infty}\in\varSigma_{\!K}^+$ be an arbitrary nonchaotic switching signal. Let $k$ be given by Lemma~\ref{lemII.1}. Since $\rho(S_k)<1$, by Lemma~\ref{lemII.2} we have some $m\ge1$ such that $\|S_k^m\|<1$.

Thus, for an arbitrary $\varepsilon>0$ there is an $\ell\ge1$ such that $\|S_k^{m\ell}\|<\varepsilon$. From Lemma~\ref{lemII.1}, it follows that as $n\to+\infty$,
\begin{equation*}
\|S_{\sigma_{\!n}}\dotsm S_{\sigma_{\!1}}\|\le\|S_{\sigma_{\!n_{m\ell}+m\ell}}\dotsm S_{\sigma_{\!n_{m\ell}+1}}\|<\varepsilon.
\end{equation*}
So, $\|S_{\sigma_{\!n}}\dotsm S_{\sigma_{\!1}}\|\to0$ as $n\to+\infty$, since $\varepsilon>0$ is arbitrary. This completes the proof of Theorem~\ref{thmII.3}.
\end{proof}

Under condition (\ref{eq1.3a}), the statement of Theorem~\ref{thmII.3} will be strengthened by Corollary~\ref{corV.3} in Section~\ref{secV}.
%%%%%%%%%%%%%%%%%%%%%%%%%%%%%%%%%%%%%%%%%%%%%%%%%%%%%%%%%%%%%%%%%%%%%%%%%%%
%%%%%%%%%%%%%%%%%%%%%%%%%%%%%%%%%%%%%%%%%%%%%%%%%%%%%%%%%%%%%%%%%%%%%%%%%%%
\section{$\omega$-limit sets for product bounded systems}\label{secIII}%%%

In this section, we will introduce $\omega$-limit sets and give some notes on our splitting theorem stated in Section~\ref{secI.3} that is very important for our arguments in the next sections.

\subsection{$\omega$-limit sets of a trajectory}

We now consider the linear inclusion (\ref{eq1.1}) generated by $\bS=\{S_1,\dotsc,S_K\}\subset\mathbb{R}^{d\times d}$ where $2\le K<+\infty$, as in Section~\ref{secI}. The classical one-sided Markov shift transformation
\begin{equation*}
\theta\colon\varSigma_{\!K}^+\rightarrow\varSigma_{\!K}^+
\end{equation*}
is defined as
\[
 \sigma=(\sigma_{\!n})_{n=1}^{+\infty}\mapsto\theta(\sigma)=(\sigma_{\!n+1})_{n=1}^{+\infty}\qquad\forall
\sigma\in\varSigma_{\!K}^+.
\]

\begin{defn}[\cite{RSD:aut,SVR:IEEE,Bal}]\label{defIII.1}%%%
Let $x_0\in\mathbb{R}^d$ be an initial state and
$\sigma=(\sigma_{\!n})_{n=1}^{+\infty}\in\varSigma_{\!K}^+$ a
switching signal. The set of all limit points of the sequence $\langle S_{\sigma_{\!n}}\dotsm
S_{\sigma_{\!1}}(x_0)\rangle_{n=1}^{+\infty}$ in $\mathbb{R}^d$ is called the $\omega$-limit set of $\bS$ at the input $(x_0,\sigma)$. We denote it by $\omega(x_0,\sigma)$ here.
\end{defn}

It is easy to see that for any switching signal $\sigma$, the
corresponding switched system is asymptotically stable if and only
if $\omega(x_0,\sigma)=\{0\}\,\forall x_0\in\mathbb{R}^d$. Thus
we need to consider the structure of $\omega(x_0,\sigma)$ in order to
study the stability of the switched dynamics induced by $\bS$.

\begin{lem}\label{lemIII.2}%%%
Assume $\bS$ is product bounded; that is, there is a matrix
norm  $\|\cdot\|$ on $\mathbb{R}^{d\times d}$ such that
$\|S_k\|\leq 1$ for all $1\le k\le K$.
Then, for any initial data $x_0\in\mathbb{R}^d$ and any switching
signal $\sigma$, the following two statements hold.
\begin{enumerate}
\item[(1)] The $\omega$-limit set $\omega(x_0,\sigma)$ is a compact subset contained in a
sphere $\{x\in\mathbb{R}^d;\ \|x\|=r\}$, for some $r\ge 0$.

\item[(2)] The trajectory $\langle x_n(x_0,\sigma)\rangle_{n=1}^{+\infty}$ in $\mathbb{R}^d$ tends to $0$ as $n\rightarrow\infty $ if and only if
there exists a subsequence of it which tends to $0$.
\end{enumerate}
\end{lem}

\begin{proof}
Since the sequence $\langle \|S_{\sigma_{\!n}}\dotsm
S_{\sigma_{\!1}}(x_0)\|\rangle_{n=1}^{+\infty}$ is nonincreasing in $\mathbb{R}$ for any $\sigma\in\varSigma_{\!K}^+$, it is convergent as $n\rightarrow+\infty$. Denoted
 by $r$ its limit, we have the statement (1). The statement (2) follows immediately from the statement (1). This proves Lemma~\ref{lemIII.2}.
\end{proof}

In the case (2) of this lemma, we call the orbit $\langle x_n(x_0,\sigma)\rangle_{n=1}^{+\infty}$ with initial
value $x_0$ is asymptotically stable.

We note here that Lemma~\ref{lemIII.2} is actually proved in \cite{SVR:IEEE, Bal} for the continuous-time case, but \cite{Bal} is under the condition that each member of $\bS$ shares a common, not strict in general, quadratic Lyapunov function and \cite{SVR:IEEE} under an additional assumption of ``paracontraction" except the Lyapunov function. In Section~\ref{secIII.3}, we will consider the $\omega$-limit set of a matrix trajectory $\langle S_{\sigma_{\!n}}\dotsm S_{\sigma_{\!1}}\rangle_{n=1}^{+\infty}$. In addition, in the continuous-time case, $\omega(x_0,\sigma)$ is a connected set. This is an important property needed in \cite{SVR:IEEE, Bal}.

For a given switching signal, to consider the stability of the
corresponding switched system, we need to classify which kind of
initial values in $\mathbb{R}^d$ makes the corresponding orbits
asymptotically stable. It is difficult to have such classification
for a general switching signal. In the following, for the recurrent
switching signal, we have a classification result.

%%%%%%%%%%%%%%%%%%%%%%%%%%%%%%%%%%%%%%%%%%%%%%%%%%%%%%%%%%
\subsection{Decomposition for general extremal norm}\label{secIII.2}%%%%

In this subsection, we will introduce a preliminary splitting theorem of the state space $\mathbb{R}^d$ which
plays the key in our classification.

First, we recall from \cite{NS60, Wal82} that for a topological dynamical system $T\colon\Omega\rightarrow\Omega$ on a separable metrizable space $\Omega$,
a point $w\in\Omega$ is called ``recurrent", provided that one
can find a positive integer sequence $n_i\nearrow+\infty$ such that $T^{n_i}(w)\to w$ as $i\to+\infty$.
And $w\in\Omega$ is said to be ``weakly Birkhoff recurrent" \cite{Zhou92} (also see \cite{Dai-JDE}), provided that for any $\varepsilon >0$, there exists
an integer $N_\varepsilon>1$ such that
\[
    \sum_{i=0}^{jN_\varepsilon-1}I_{\mathbf{B}(w,\varepsilon)}(T^i(w))\geq
    j\qquad\forall j\in\mathbb{N},
\]
where $I_{\mathbf{B}(w,\varepsilon)}\colon\Omega\to \{0,1\}$ is the
characteristic function of the open ball
$\mathbf{B}(w,\varepsilon)$ of radius $\varepsilon$ centered at
$w$ in $\Omega$. We denote by $R(T)$ and $W(T)$, respectively, the set of all
recurrent points and weakly Birkhoff recurrent points of $T$. It is easy to
see that $R(T)$ and $W(T)$ both are invariant under $T$ and
$W(T)\subset R(T)$.

In the qualitative theory of ordinary differential equation, this
type of recurrent point is also called a ``Poisson stable" motion,
for instance, in \cite{NS60}.

For the one-sided Markov shift $(\varSigma_{\!K}^+, \theta)$, it is
easily checked that every periodically switched signal is recurrent.
And $\sigma=(\sigma_{\!n})_{n=1}^{+\infty}\in R(\theta)$ means that there
exists a subsequence $n_i\nearrow+\infty$ such that
$\theta^{n_i}(\sigma)\to\sigma$ as $i\to+\infty$. This implies that
\begin{equation*}
S_{\sigma_{n_i+n}}\cdots S_{\sigma_{n_i+1}}\to S_{\sigma_n}\cdots S_{\sigma_1}\quad \textrm{as }i\to+\infty
\end{equation*}
for any $n\ge1$. We should note that for any two finite-length words $w\not=w^\prime$, the switching signal $\sigma=(w^\prime, w,w,w,\dotsc)$ is not recurrent.

For any function $A\colon\Omega\rightarrow\mathbb{R}^{d\times d}$, the cocycle $A_T\colon\mathbb{N}\times\Omega\rightarrow\mathbb{R}^{d\times d}$ driven by $T$ is defined as
\begin{equation*}
A_T(n,w)=A(T^{n-1}w)\cdots A(w)
\end{equation*}
for any $n\ge1$ and all $w\in\Omega$.
Now, our basic decomposition theorem can be stated as follows:

\begin{theorem}[{\cite[Theorem~B$^\prime$]{DHX-AIM}}]\label{thmIII.3}%%%
Let $T\colon\Omega\rightarrow \Omega$
be a continuous transformation of a
separable metrizable space $\Omega$. Let
$A\colon\Omega\rightarrow\mathbb{R}^{d\times d}$ be a continuous
family of matrices with the property that there exists a norm
$\|\cdot\|$ such that
\begin{equation*}
\|A_T(n,w)\|\le 1\quad \forall n\ge1\textrm{ and
}w\in\Omega.
\end{equation*}
Then for any recurrent point $w$ of $T$, there corresponds a splitting of $\mathbb{R}^d$ into subspaces
\begin{equation*}
\mathbb{R}^d=E^s(w)\oplus E^c(w),
\end{equation*}
such that
\begin{align*}
&\lim_{n\to+\infty}\|A_T(n,w)(x)\|=0&&\forall x\in E^s(w)\\
\intertext{and} &\|A_T(n,w)(x)\|=\|x\|\;\forall n\geq
1&& \forall x\in E^c(w).
\end{align*}
\end{theorem}

Here $\|\cdot\|$ does not need to be a Lyapunov norm $\|\cdot\|_P$ as in (\ref{eq1.3b}) and further the central manifold $E^c(\sigma)$ is not necessarily unique and invariant. Although $\|A_T(n,w)|E^s(w)\|$ converges to $0$, yet $\|A_T(n,w)|E^s(w)\|$ does not need to converge exponentially fast, as is shown by \cite[Example 4.6]{DHX-AIM}.

However, under the assumptions of Theorem~\ref{thmIII.3}, if $w$ is a weakly Birkhoff recurrent point of $T$, we
have the following alternative results:

\begin{theorem}\label{thmIII.4}%%%
Let $T\colon\Omega\rightarrow \Omega$
be a continuous transformation of a
separable metrizable space $\Omega$. Let
$A\colon\Omega\rightarrow\mathbb{R}^{d\times d}$ be a continuous
family of matrices with the property that there exists a norm
$\|\cdot\|$ such that $\|A_T(n,w)\|\le 1$ for all $n\ge1$ and
$w\in\Omega$. If $w\in\Omega$ is a weakly Birkhoff recurrent point of $T$,
Then either
\begin{align*}
&\|A_T(n,w)\|\xrightarrow[]{\textrm{exponentially
fast}}0&&\textrm{ as }n\to+\infty,\\
\intertext{or}
&\|A_T(n,T^i(w))\|=1\; \forall i\ge0&&\textrm{for
}n\ge1.
\end{align*}
\end{theorem}

\begin{proof}
If there exist $i\ge0$ and $n\ge1$ such that $\|A_T(n,T^i(w))\|<1$
then from $T^i(w)\in W(T)$ and \cite[Theorem~2.4]{Dai-JDE}, it follows that
\begin{equation*}
\|A_T(m,T^i(w))\|\xrightarrow[]{\textrm{exponentially
fast}}0\qquad\textrm{ as }m\to+\infty.
\end{equation*}
This completes the proof of Theorem~\ref{thmIII.4}.
\end{proof}

%%%%%%%%%%%%%%%%%%%%%%%%%%%%%%%%%%%%%%%%%%%%%%%%%%%%%%%%%%%%%%%%%%%%%%%%%
\subsection{Decomposition under a weak Lyapunov matrix}\label{secIII.3}%
For a recurrent switching
signal $\sigma=(\sigma_{\!n})_{n=1}^{+\infty}$ of $\bS$, to consider its stability,
it is essential to compute the stable manifold $E^s(\sigma)$. From
the proof of Theorem~\ref{thmIII.3} presented in \cite{DHX-AIM}, we know that $E^s(\sigma)$ is the
kernel of an idempotent matrix that is a limit point of
$S_{\sigma_{n_i}}\cdots S_{\sigma_{\!1}}$ with $\theta^{n_i}(\sigma)\to\sigma$ as
$i\to +\infty$.

However, in applications, it is not easy to identify which subsequence
$\langle n_i\rangle_{i\ge1}$ with this property. In this subsection, instead of
the product boundedness, we assume the more strong condition (\ref{eq1.3a}) with induced norm $\|\cdot\|_P$ on $\mathbb{R}^d$.

In this case, we can calculate the stable manifold $E^s(\sigma)$ for
any switching signal $\sigma$ (not necessarily recurrent) of $\bS$. To
do this end, we first consider the geometry of the limit sets $\omega(x_0,\sigma)$ of
$\bS$ driven by $\sigma$. For the similar results in continuous-time switched linear
systems, see \cite{Bal}.

For any switching signal $\sigma=(\sigma_{\!n})_{n=1}^{+\infty}\in\varSigma_{\!K}^+$, on the other hand,
we will consider the sequence $\langle S_{\sigma_{\!n}}\cdots S_{\sigma_{\!1}}\rangle_{n=1}^{+\infty}$
of matrices and let $\omega(\sigma)$ denote the set of all limit
points of this sequence in $\mathbb{R}^{d\times d}$.

\begin{defn}[\cite{Wir02,Bal}]\label{deIII.5}%%
The set $\omega(\sigma)$ is called the $\omega$-limit set of $\bS$ driven by $\sigma$, for any $\sigma\in\varSigma_{\!K}^+$.
\end{defn}

From condition (\ref{eq1.3a}), it follows immediately
that $\omega(\sigma)$ is non-empty and compact. But it may not be a
semigroup in the sense of matrix multiplication when $\sigma$ is not
a recurrent switching signal. We note that if $\sigma\in R(\theta)$ then from the proof of \cite[Theorem~4.2]{DHX-AIM}, $\omega(\sigma)$ contains a nonempty compact semigroup and so there is an idempotent element in $\omega(\sigma)$.

Parallel to Lemma~\ref{lemIII.2}, we can obtain the following result.

\begin{lem}\label{lemIII.6}%%
Under condition (\ref{eq1.3a}), there follows the following statements.
\begin{enumerate}
\item[(a)] For any switching signal $\sigma\in\varSigma_{\!K}^+$ of $\bS$, it holds that
\begin{equation*}
\omega(\sigma)\subset\{M\in\mathbb{R}^{d\times d}\colon \|M\|_P=r\},
\end{equation*}
for some constant $0\le r\le 1$; if $\sigma$ is further recurrent, then either $r=0$ or $1$.

\item[(b)] For any input $(x_0,\sigma)\in\mathbb{R}^d\times\varSigma_{\!K}^+$ for $\bS$, we have
\begin{equation*}
   \omega(x_0,\sigma)=\{M(x_0)\,|\,M\in\omega(\sigma)\}=\omega(\sigma)(x_0).
\end{equation*}

\item[(c)] For any two elements $M$ and $N$ in $\omega(\sigma)$,
it holds that
 \[
   M^TPM=N^TPN.
 \]
\end{enumerate}
\end{lem}

We note that the continuous-time cases of the statements (b) and (c) of Lemma~\ref{lemIII.6} have been proved in \cite[$\S$3]{Bal} using the polar decomposition of matrices. We here present a simple treatment for the sake of self-closeness.

\begin{proof}
We first note that from (\ref{eq1.3a}) and (\ref{eq1.3b}), it follows immediately that $\|S_k\|_P\le1$ for all indices $1\le k\le K$.

For the statement (b), we let $(x_0,\sigma)\in\mathbb{R}^d\times\varSigma_{\!K}^+$ be arbitrary. If $M\in\omega(\sigma)$, it is
clear that $M(x_0)\in\omega(x_0,\sigma)$. Conversely, let
$y\in\omega(x_0,\sigma)$ be arbitrary. By the definition of $\omega(x_0,\sigma)$ there exists an increasing
sequence $\{n_i\}$ such that
 \[
   y=\lim_{i\to\infty}S_{\sigma_{n_i}}\cdots S_{\sigma_{\!1}}(x_0).
 \]
The product boundedness condition implies that the sequence $\langle S_{\sigma_{n_i}}\cdots S_{\sigma_{\!1}}\rangle_{i=1}^{+\infty}$
has a convergent subsequence, whose limit element is denoted by
$M$. Thus $y=M(x_0)$.

For the statement (c) of Lemma~\ref{lemIII.6}, let $M,\ N\in\omega(\sigma)$ be arbitrary. As $\|S_k\|_P\le1$ for all $1\le k\le K$, from Lemma~\ref{lemIII.2} we
have
\begin{equation*}
   \|M(x)\|_P=\|N(x)\|_P\quad\forall x\in\mathbb{R}^d.
\end{equation*}
That is
\begin{equation*}
   x^T(M^TPM-N^TPN)x=0\quad\forall x\in\mathbb{R}^d.
\end{equation*}
It follows, from the symmetry of the matrix $M^TPM-N^TPN$, that
\begin{equation*}
    M^TPM=N^TPN.
\end{equation*}
This proves the statement (c) of Lemma~\ref{lemIII.6}.

Finally, the statement (a) of Lemma~\ref{lemIII.6} comes from the statement (c) and Theorem~\ref{thmIII.3}. In fact, let $M,\ N\in\omega(\sigma)$ be arbitrary. Then there are vectors $x,y\in\mathbb{R}^d$ such that
\begin{equation*}
\|x\|_P=\|y\|_P=1,\quad\|M\|_P=\|M(x)\|_P,\quad \textrm{and}\quad\|N\|_P=\|N(y)\|_P.
\end{equation*}
So, from (c) it follows that
\begin{equation*}
\|M\|_P=\sqrt{x^TM^TPMx}=\sqrt{x^TN^TPNx}\le\|N\|_P=\sqrt{y^TN^TPNy}=\sqrt{y^TM^TPMy}\le\|M\|_P.
\end{equation*}
This together with Theorem~\ref{thmIII.3} proves the statement (a) of Lemma~\ref{lemIII.6}.

Thus the proof of Lemma~\ref{lemIII.6} is completed.
\end{proof}

Let $M\in\omega(\sigma)$.
Then $\sqrt{M^TPM}$ is a nonnegative-definite matrix which does not
depend on the choice of the matrix $M\in\omega(\sigma)$ by the statement (c) of
Lemma~\ref{lemIII.6} and is uniquely decided by the switching signal
$\sigma$. So, we write
\begin{equation}\label{eqIII.1}%%
Q_\sigma=\sqrt{M^TPM}\quad\forall M\in\omega(\sigma).
\end{equation}

The continuous-time case of the following statement (1) of Proposition~\ref{propIII.7} has already been proved by Balde and Jouan~\cite[Theorem~1]{Bal} using the polar decomposition of matrices.

\begin{prop}\label{propIII.7}%%%
Under condition (\ref{eq1.3a}), for any switching signal
$\sigma=(\sigma_{\!n})_{n=1}^{+\infty}$ of $\bS$, there hold the following two statements.
\begin{enumerate}
 \item[(1)] The switching signal $\sigma$ is asymptotically stable for $\bS$; that is,
 \[
   \lim_{n\to\infty}S_{\sigma_{\!n}}\dotsm S_{\sigma_{\!1}}(x_0)=0\quad\forall
   x_0\in\mathbb{R}^d,
 \]
 if and only if $Q_{\sigma}=0$;

\item[(2)] If $Q_{\sigma}\neq 0$, then
\begin{equation*}
\lim_{n\to+\infty}\|S_{\sigma_{\!n}}\dotsm S_{\sigma_{\!1}}(x_0)\|_P=\|Q_{\sigma}(x_0)\|_2\quad\forall x_0\in\mathbb{R}^d.
\end{equation*}
So, the stable manifold of $\bS$ at $\sigma$ is such that $E^s(\sigma)=$ kernel of $Q_{\sigma}$; that is
\begin{equation*}
   \lim_{n\to+\infty}\|S_{\sigma_{\!n}}\dotsm S_{\sigma_{\!1}}(x_0)\|_P=0\quad\forall
   x_0\in E^s(\sigma).
\end{equation*}
Here $\|\cdot\|_2$ denotes the Euclidean vector norm on $\mathbb{R}^d$.
\end{enumerate}
\end{prop}

\begin{proof}
The statement (1) holds trivially from the statement (a) of Lemma~\ref{lemIII.6} or from the statement (2) to be proved soon. We next will prove the statement (2). For that, let $Q_{\sigma}\neq 0$.
For an arbitrary $x_0\in\mathbb{R}^d$, by the
definition of $Q_\sigma$ as in (\ref{eqIII.1}) there exists a subsequence
$\langle n_i\rangle_{i\ge1}$ and some $M\in\omega(\sigma)$ such that
\[
\lim_{i\to+\infty}\|S_{\sigma_{n_i}}\cdots S_{\sigma_{\!1}}(x_0)\|_P=\|M(x_0)\|_P=\sqrt{x_0^TQ_\sigma^2x_0}=\sqrt{x_0^TQ_\sigma^TQ_\sigma x_0}=\|Q_\sigma(x_0)\|_2.
\]
Therefore, by (\ref{eq1.3a}) we have
\[
 \lim_{n\to+\infty}\|S_{\sigma_{\!n}}\dotsm S_{\sigma_{\!1}}(x_0)\|_P=\|Q_{\sigma}(x_0)\|_2.
\]
This thus completes the proof of Proposition~\ref{propIII.7}.
\end{proof}

We note here that if $Q_\sigma$ is idempotent, then from Proposition~\ref{propIII.7} we have $E^c(\sigma)=\mathrm{Im}(Q_\sigma)$ and $\mathbb{R}^d=E^s(\sigma)\oplus E^c(\sigma)$. Because in general there lacks the recurrence of $\sigma$, one cannot define a central manifold $E^c(\sigma)$ satisfying $\mathbb{R}^d=E^s(\sigma)\oplus E^c(\sigma)$
as done in Theorem~\ref{thmIII.3}. However, we will establish another type of splitting theorem in the case $d=2$ for $\bS$ driven by a general switching signal, not necessarily recurrent.

For that, we first introduce several notations for the sake of our convenience. For any given $A\in\mathbb{R}^{d\times d}$ and any vector norm $\|\cdot\|$ on $\mathbb{R}^d$, write
\begin{equation}\label{eqIII.2}
\|A\|_{\mathrm{co}}=\min\left\{\|A(x)\|\colon x\in\mathbb{R}^d\textrm{ with }\|x\|=1\right\},
\end{equation}
called the \emph{co-norm} (also \emph{minimum norm} in some literature) of $A$ under $\|\cdot\|$.

\begin{defn}\label{defnIII.8}%%%
Under condition (\ref{eq1.3a}), for any switching signal $\sigma\in\varSigma_{\!K}^+$ the numbers
\begin{equation*}
r_E(\sigma)^{}:=\|M\|_P\quad\textrm{and}\quad r_I(\sigma)^{}:=\|M\|_{P,\mathrm{co}},
\end{equation*}
for $M\in\omega(\sigma)$, are called the $\omega$-exterior and $\omega$-interior radii of $\bS$ driven by $\sigma$, respectively.
\end{defn}

According to the statement (c) of Lemma~\ref{lemIII.6}, $r_E(\sigma)^{}$ and $r_I(\sigma)^{}$ both are well defined independent of the choice of $M$.

Motivated by $\mathbb{E}^c(\sigma)$ in \cite[$\S$5.2.2]{Dai-JDE} and by $\mathcal{V}_i$ in \cite[Lemma~1]{Bal}, for any given $A\in\mathbb{R}^{d\times d}$ and any vector norm $\|\cdot\|$ on $\mathbb{R}^d$, let
\begin{subequations}\label{eqIII.3}
\begin{align}
&\mathbb{K}_{\|\cdot\|}(A)=\left\{x\in\mathbb{R}^d\colon\|A(x)\|=\|A\|\cdot\|x\|\right\}\\
\intertext{and}&\mathbb{K}_{\|\cdot\|_{\mathrm{co}}}(A)=\left\{x\in\mathbb{R}^d\colon\|A(x)\|=\|A\|_{\mathrm{co}}\cdot\|x\|\right\}.
\end{align}
\end{subequations}
Clearly, if $\ker(A)\not=\{0\}$, then $\|A\|_{\mathrm{co}}=0$ and so $\mathbb{K}_{\|\cdot\|_{\mathrm{co}}}(A)=\ker(A)$ in this case.

For a general norm $\|\cdot\|$ on $\mathbb{R}^d$, $\mathbb{K}_{\|\cdot\|_{\mathrm{co}}}(A)$ and $\mathbb{K}_{\|\cdot\|}(A)$ are not necessarily linear subspaces. However, for a Lyapunov norm, we can obtain the following.

\begin{lem}\label{lemIII.9}%%%
Under the Lyapunov norm $\|\cdot\|_P$ as in (\ref{eq1.3b}), there $\mathbb{K}_{\|\cdot\|_{P,\mathrm{co}}}(A)$ and $\mathbb{K}_{\|\cdot\|_P}(A)$ both are linear subspaces of $\mathbb{R}^d$ for any $A\in\mathbb{R}^{d\times d}$.
\end{lem}

\begin{proof}
Let $A\in\mathbb{R}^{d\times d}$ be arbitrarily given. By definitions, we have
\begin{equation*}
\begin{split}
 x\in\mathbb{K}_{\|\cdot\|_P}(A)&\Leftrightarrow x^T\|A\|_PPx-x^TA^TPAx=0\\
 &\Leftrightarrow x^T(\|A\|_PP-A^TPA)x=0\\
 &\Leftrightarrow \|G(x)\|_2=0\\
 &\Leftrightarrow x\in \ker(G).
\end{split}
\end{equation*}
Here $G^2=\|A\|_PP-A^TPA\ge0$ is symmetric. Since $\ker(G)$, the kernel of $x\mapsto Gx$, is a linear
subspace of $\mathbb{R}^d$, $\mathbb{K}_{\|\cdot\|_P}$ is also a linear subspace of $\mathbb{R}^d$.

On the other hand, for any $x\in\mathbb{R}^d$ we have $\|A(x)\|_P\ge\|A\|_{P,\mathrm{co}}\cdot\|x\|_P$. So,
\begin{equation*}
x^T(A^TPA-\|A\|_{P,\mathrm{co}}P)x\ge0\quad\forall x\in\mathbb{R}^d.
\end{equation*}
Let $H^2=A^TPA-\|A\|_{P,\mathrm{co}}P$, which is symmetric and nonnegative-definite. Then it holds that $\mathbb{K}_{\|\cdot\|_{P,\mathrm{co}}}(A)=\ker(H)$, a linear subspace.

Thus, the proof of Lemma~\ref{lemIII.9} is completed.
\end{proof}

Now, the improved splitting theorem can be stated as follows:

\begin{theorem}\label{thmIII.10}%%%
Let $\bS=\{S_1,\dotsc,S_K\}\subset\mathbb{R}^{2\times 2}$ satisfy condition (\ref{eq1.3a}). Then, for any switching signal $\sigma\in\varSigma_{\!K}^+$, not necessarily recurrent, there exists a splitting of $\mathbb{R}^2$ into subspaces
\begin{equation*}
\mathbb{R}^2=\mathbb{K}_{\|\cdot\|_{P,\mathrm{co}}}(\sigma)\oplus\mathbb{K}_{\|\cdot\|_{P}}(\sigma)
\end{equation*}
such that
\begin{align*}
&\lim_{n\to+\infty}\|S_{\sigma_{\!n}}\dotsm S_{\sigma_{\!1}}(x_0)\|_P=r_I^{}\|x_0\|_P&&\forall x_0\in\mathbb{K}_{\|\cdot\|_{P,\mathrm{co}}}(\sigma),\\
&\lim_{n\to+\infty}\|S_{\sigma_{\!n}}\dotsm S_{\sigma_{\!1}}(x_0)\|_P=r_E^{}\|x_0\|_P&&\forall x_0\in\mathbb{K}_{\|\cdot\|_{P}}(\sigma),\\
\intertext{and}
&r_I\|x_0\|_P<\lim_{n\to+\infty}\|S_{\sigma_{\!n}}\dotsm S_{\sigma_{\!1}}(x_0)\|_P<r_E^{}\|x_0\|_P&&\forall x_0\in\mathbb{R}^2-\mathbb{K}_{\|\cdot\|_{P,\mathrm{co}}}(\sigma)\cup\mathbb{K}_{\|\cdot\|_{P}}(\sigma).
\end{align*}
\end{theorem}

\begin{proof}
Let $r_I<r_E$ and $M\in\omega(\sigma)$. Define $\mathbb{K}_{\|\cdot\|_{P,\mathrm{co}}}(\sigma)=\mathbb{K}_{\|\cdot\|_{P,\mathrm{co}}}(M)$ and $\mathbb{K}_{\|\cdot\|_{P}}(\sigma)=\mathbb{K}_{\|\cdot\|_{P}}(M)$. From the statement (2) of Proposition~\ref{propIII.7}, it follows that
$\mathbb{K}_{\|\cdot\|_{P,\mathrm{co}}}(\sigma)$ and $\mathbb{K}_{\|\cdot\|_{P}}(\sigma)$ both are independent of the choice of $M$. So, $\mathbb{R}^2=\mathbb{K}_{\|\cdot\|_{P,\mathrm{co}}}(\sigma)\oplus\mathbb{K}_{\|\cdot\|_{P}}(\sigma)$ from Lemma~\ref{lemIII.9}. We note that if $r_I=r_E$, then $\mathbb{K}_{\|\cdot\|_{P,\mathrm{co}}}(\sigma)=\mathbb{K}_{\|\cdot\|_{P}}(\sigma)=\mathbb{R}^2$. This completes the proof of Theorem~\ref{thmIII.10}.
\end{proof}

In the case where $\sigma$ is recurrent, one can easily see that
\begin{equation*}
E^s(\sigma)=\mathbb{K}_{\|\cdot\|_{P,\mathrm{co}}}(\sigma)\quad\textrm{and}\quad E^c(\sigma)=\mathbb{K}_{\|\cdot\|_{P}}(\sigma).
\end{equation*}
%%%%%%%%%%%%%%%%%%%%%%%%%%%%%%%%%%%%%%%%%%%%%%%%%%%%%%%%%%%%
%%%%%%%%%%%%%%%%%%%%%%%%%%%%%%%%%%%%%%%%%%%%%%%%%%%%%%%%%%%%
\section{Asymptotical stability under a weak Lyapunov matrix}\label{secIV}%%%

In this section, we will discuss the stability of switched linear system
with a common, but not necessarily strict, quadratic Lyapunov function. In this case, a criteria for stability is
derived without computing the limit matrix $Q_{\sigma}$ as in (\ref{eqIII.1}). We still assume $\bS$ is composed
of finitely many subsystems. That is, $\bS=\{S_1,\dotsc, S_K\}$
with $2\le K<+\infty$.

%%%%%%%%%%%%%%%%%%%%%%%%%%%%%%%%%%%%%%%%%%%%%%%%%%%%%
\subsection{Stability of generic recurrent switching signals}\label{secIV.1}%%%
Now for $\sigma=(\sigma_{\!n})_{n=1}^{+\infty}\in\varSigma_{\!K}^+$, if $\mathrm{Card}\{n\,|\,\sigma_{\!n}=k\}=\infty$ for all $1\le k\le K$ then $\sigma$ is called ``generic." Recall that a switching signal $\sigma=(\sigma_{\!n})_{n=1}^{+\infty}\in\varSigma_{\!K}^+$ is said to be \emph{stable} for $\bS$ if
\[
  \|S_{\sigma_{\!n}}\cdots S_{\sigma_{\!1}}\|\rightarrow 0\quad \textrm{as }n\to+\infty.
\]
(Note that the stability is independent of the chosen norm $\|\cdot\|$.)
As is known, a switching system which is asymptotically stable for
all periodically switching signals does not need to be
asymptotically stable for all switching signals in general~\cite{BM,BTV,Koz07,HMST}. However we
can obtain the following result.

\begin{lem}\label{lemIV.1}%%
If all recurrent switching signals are stable for
$\bS$, then it is asymptotically stable driven by all switching
signals in $\varSigma_{\!K}^+$.
\end{lem}

\begin{proof}
Since the set $R(\theta)$ of all recurrent switching signals has
full measure $1$ for all ergodic measures with respect to
$(\varSigma_{\!K}^+,\theta)$, the result follows from \cite[Lemma~2.3]{DHX-ERA}.
\end{proof}

By Lemma~\ref{lemIV.1}, to obtain the asymptotic stability of $\bS$,
it suffices to prove that it is only asymptotically stable driven by all recurrent
switching signals.

In addition, we need the following lemma.

\begin{lem}\label{lemIV.2}%%%%
Under condition (\ref{eq1.3a}), if $\|S_k\|_P=1$ and $\mathbb{K}_{\|\cdot\|_P}(S_k)$ is $S_k$-invariant, then $\rho(S_k)=1$.
\end{lem}

Here $\mathbb{K}_{\|\cdot\|_P}(S_k)$ is defined as in (\ref{eqIII.3}).

\begin{proof}
The statement comes obviously from Lemma~\ref{lemIII.9}.
\end{proof}

In the following, for simplicity, we just consider a switched system
which is composed of two subsystems. That is, $K=2$.

\begin{lem}\label{lemIV.3}%%%
Under condition (\ref{eq1.3a}) with $K=2$ (i.e., $\bS=\{S_1,S_2\}$), if $\|S_1\|_P=\|S_2\|_P=1$ and
\begin{equation}\label{eqIV.1}
\mathbb{K}_{\|\cdot\|_P}(S_1)\cap\mathbb{K}_{\|\cdot\|_P}(S_2)=\{0\},
\end{equation}
and at least one of them is invariant (i.e., $S_1(\mathbb{K}_{\|\cdot\|_P}(S_1))=\mathbb{K}_{\|\cdot\|_P}(S_1)$ or $S_2(\mathbb{K}_{\|\cdot\|_P}(S_2))=\mathbb{K}_{\|\cdot\|_P}(S_2)$),
then every generic switching signal is stable for $\bS$.
\end{lem}

\begin{proof}
Assume that $\mathbb{K}_{\|\cdot\|_P}(S_1)$ is $S_1$-invariant. (Otherwise, if $\mathbb{K}_{\|\cdot\|_P}(S_2)$ is $S_2$-invariant,
the proof is the same.) Let $\sigma=(\sigma_{\!n})_{n=1}^{+\infty}$ be a generic
switching signal; that is, in $(\sigma_{\!n})_{n=1}^{+\infty}$, both $1$ and $2$ appear infinitely many
times. Then there exists a subsequence $\{\sigma_{n_i}\}$ such that
\[
 \sigma_{n_i}=1\quad\textrm{and}\quad \sigma_{n_i+1}=2\quad\forall i\ge1.
\]

For a given initial value $x_0\in\mathbb{R}^d$, consider the
subsequence $\{S_{\sigma_{n_i-1}}\dotsm S_{\sigma_{\!1}}(x_0)\}_{i=1}^{+\infty}$. By
the assumption (\ref{eq1.3a}), it has a convergent subsequence in
$\mathbb{R}^d$. Without loss of generality, we assume that
\[
  S_{\sigma_{n_i-1}}\cdots S_{\sigma_{\!1}}(x_0)\rightarrow y\in\mathbb{R}^d\quad \textrm{as }
  i\rightarrow +\infty.
\]
Thus
\begin{equation*}
S_{\sigma_{n_i}}S_{\sigma_{n_i-1}}\dotsm S_{\sigma_{\!1}}(x_0)\to S_1(y),
\end{equation*}
\begin{equation*}
S_{\sigma_{n_i+1}}S_{\sigma_{n_i}}S_{\sigma_{n_i-1}}\dotsm S_{\sigma_{\!1}}(x_0)\to S_2S_1(y),
\end{equation*}
as $i\to+\infty$.
By the statement (1) of Lemma~\ref{lemIII.2}, we have
\begin{equation*}
   \|S_2S_1(y)\|_P=\|S_1(y)\|_P=\|y\|_P.
\end{equation*}
Thus $y\in\mathbb{K}_{\|\cdot\|_P}(S_1)$ and $S_1(y)\in\mathbb{K}_{\|\cdot\|_P}(S_2)$. From the
$S_1$-invariance of $\mathbb{K}_{\|\cdot\|_P}(S_1)$ it follows that
\[
   S_1(y)\in \mathbb{K}_{\|\cdot\|_P}(S_1)\cap\mathbb{K}_{\|\cdot\|_P}(S_2).
\]
So $S_1(y)=0$ and so is $y$. From the statement (2) of Lemma~\ref{lemIII.2}, we have
\[
 S_{\sigma_{\!n}}\cdots S_{\sigma_{\!1}}(x_0)\to0\quad \textrm{as }
  n\to+\infty.
\]
That is, $\sigma$ is a stable switching signal for $\bS$. This proves Lemma~\ref{lemIV.3}.
\end{proof}

Both $\mathbb{E}^c(\sigma)$ in \cite[$\S$5.2.2]{Dai-JDE} and $\mathcal{V}_i$ in \cite[Lemma~1]{Bal} are invariant. Unfortunately, here our subspace $\mathbb{K}_{\|\cdot\|_P}(S_k)$ does not need to be
$S_k$-invariant in general. See Example~\ref{exaVI.2} in
Section~\ref{secVI}. If this is the case, we still have, however, the following criterion.

\begin{theorem}\label{thmIV.4}%%%
Under conditions (\ref{eq1.3a}) and (\ref{eqIV.1}) with $\bS=\{S_1,S_2\}\subset\mathbb{R}^{d\times d}$, the following two statements hold.
\begin{enumerate}
\item[(1)] If $d=2$, then all generic recurrent switching signals $\sigma\in\varSigma_2^+$, which satisfy
\begin{equation*}
\sigma\not=(\widehat{1,2},\widehat{1,2},\dotsc),
\end{equation*}
are stable for $\bS$;

\item[(2)] if $d=3$, then all generic recurrent switching signals $\sigma\in\varSigma_2^+$ such that
\begin{equation*}
\sigma\not=(w,w,w,\dotsc),\qquad \textrm{where }
w\in\{(1,2), (2,1), (1,2,2), (2,1,1)\},
\end{equation*}
are stable for $\bS$.
\end{enumerate}
\end{theorem}

\begin{proof}
First, if $\|S_1\|_P<1$ or $\|S_2\|_P<1$, then every generic switching signal is stable for $\bS$ and hence the statements (1) and (2) trivially hold. So, we next assume $\|S_1\|_P=\|S_2\|_P=1$. This implies that $\dim\mathbb{K}_{\|\cdot\|_P}(S_k)\ge1$ for $k=1,2$.

For the statement (1) of Theorem~\ref{thmIV.4}, from (\ref{eqIV.1}) it follows that $\dim\mathbb{K}_{\|\cdot\|_P}(S_k)=1$ for $k=1,2$.
Let $\sigma=(\sigma_{\!n})_{n=1}^{+\infty}$ be a given generic recurrent switching
signal such that
\begin{equation}\label{eqIV.2}
\sigma(\cdot+n)\not=(\widehat{1,2},\widehat{1,2},\dotsc,\widehat{1,2},\dotsc)\quad\forall n\ge1.
\end{equation}
From Theorem~\ref{thmIII.3}, there corresponds a splitting of
$\mathbb{R}^2$ into subspaces
\begin{equation*}
\mathbb{R}^2=E^s(\sigma)\oplus E^c(\sigma),
\end{equation*}
such that
\begin{align*}
&\lim_{n\to+\infty}\|S_{\sigma_{\!n}}\cdots S_{\sigma_{\!1}}(x_0)\|_P=0&&\forall x_0\in E^s(\sigma)\\
\intertext{and} & \|S_{\sigma_{\!n}}\cdots S_{\sigma_{\!1}}(x_0)\|_P=\|x_0\|_P\;\forall
n\geq 1&& \forall x\in E^c(\sigma).
\end{align*}
To prove that $\sigma$ is a stable switching signal for $\bS$, we need to
prove that $E^c(\sigma)=\{0\}$. By the genericity of $\sigma$ and
(\ref{eqIV.2}), $\sigma$ must contains the word $(1,1,2)$ or $(2,2,1)$.
Without loss of generality, we assume that
\[
(\sigma_{\!1},\sigma_{\!2},\sigma_{\!3})=(1,1,2).
\]
Thus we have
\[
 \|S_2S_1S_1(x_0)\|_P=\|S_1S_1(x_0)\|_P=\|S_1(x_0)\|_P=\|x_0\|_P\quad\forall
 x_0\in E^c(\sigma)
\]
These imply that
\[
\{x_0,\ S_1(x_0)\}\subset\mathbb{K}_{\|\cdot\|_P}(S_1),\quad S_1S_1(x_0)\in\mathbb{K}_{\|\cdot\|_P}(S_2).
\]
Suppose that $x_0\neq 0$. It follows from $\dim\mathbb{K}_{\|\cdot\|_P}(S_1)=1$ that
there exists a real number $\lambda$ with $|\lambda|=1$ such that
\[
 S_1(x_0)=\lambda x_0.
\]
This means that $x_0$ is an eigenvector of $S_1$ with eigenvalue
$\lambda$. So
\[
S_1S_1(x_0)=\lambda^2x_0\in \mathbb{K}_{\|\cdot\|_P}(S_1).
\]
Therefore $S_1S_1(x_0)\in\mathbb{K}_{\|\cdot\|_P}(S_1)\cap\mathbb{K}_{\|\cdot\|_P}(S_2)=\{0\}$. Thus we
have $S_1S_1(x_0)=0$, which implies $x_0=0$, a contradiction.

Next, for proving the statement (2) of Theorem~\ref{thmIV.4} that $d=3$, by (\ref{eqIV.1}), we have that one of $\mathbb{K}_{\|\cdot\|_P}(S_1), \mathbb{K}_{\|\cdot\|_P}(S_2)$ has dimension $1$ and the other has dimension at least $1$ and at most $2$.

If both $\mathbb{K}_{\|\cdot\|_P}(S_1)$ and
$\mathbb{K}_{\|\cdot\|_P}(S_2)$ have dimension $1$, then by the same argument as in
the statement (1), all generic recurrent switching signals satisfying (\ref{eqIV.2})
are stable for $\bS$.

Next, we assume that, for example,
\[
 \dim\mathbb{K}_{\|\cdot\|_P}(S_1)=1\quad\textrm{and}\quad \dim\mathbb{K}_{\|\cdot\|_P}(S_2)=2.
\]
We claim that for any generic recurrent switching signal $\sigma=(\sigma_{\!n})_{n=1}^{+\infty}\in\varSigma_2^+$, if
\begin{equation}\label{eqIV.3}
\sigma(\cdot+n)\not\in\left\{(\widehat{1,2},\widehat{1,2},\dotsc,\widehat{1,2},\dotsc),(\widehat{1,2,2},\widehat{1,2,2},\dotsc,\widehat{1,2,2},\dotsc)\right\}\quad\forall n\ge1.
\end{equation}
then $\sigma$ is stable for $\bS$. There is no loss of generality in assuming $\sigma_{\!1}=1$; otherwise replacing $\sigma$ by $\sigma(\cdot+n)$ for some $n\ge1$. Then,
\begin{equation*}
\mathbb{K}_{\|\cdot\|_P}(S_1)=E^c(\sigma)\quad \textrm{if }E^c(\sigma)\not=\{0\},
\end{equation*}
where $E^c(\sigma)$ is given by Theorem~\ref{thmIII.3}.

Whenever the word $11$ appears in the sequence $(\sigma_{\!n})_{n=1}^{+\infty}$, $\mathbb{K}_{\|\cdot\|_P}(S_1)$ is $S_1$-invariant. Then, Lemma~\ref{lemIV.3} follows that $\sigma$ is stable for $\bS$. Next, we assume $11$ does not appear in $(\sigma_{\!n})_{n=1}^{+\infty}$. If $121$ appears in $(\sigma_{\!n})_{n=1}^{+\infty}$ then $\widehat{12}\widehat{12}\widehat{12}\dotsm$ must appear too, a contradiction. So, $121$ cannot appear in $(\sigma_{\!n})_{n=1}^{+\infty}$. Then $122$ must appear. If $1221$ appears in $(\sigma_{\!n})_{n=1}^{+\infty}$ then $\widehat{122}\widehat{122}\widehat{122}\dotsm$ must appear too, a contradiction. Thus, the word $1222$ must appear in $(\sigma_{\!n})_{n=1}^{+\infty}$.

When $\sigma$ contains the word $(2,2,2,1)$, assume that, for example,
\[
    (\sigma_{\!{n+1}},\sigma_{\!{n+2}},\sigma_{\!{n+3}},\sigma_{\!{n+4}})=(2,2,2,1).
\]
Then we have
\[
 \|S_1S_2S_2S_2(x_0)\|_P=\|S_2S_2S_2(x_0)\|_P=\|S_2S_2(x_0)\|_P=\|S_2(x_0)\|_P=\|x_0\|_P\quad\forall
 x_0\in E^c(\sigma(\cdot+n)),
\]
which show that for all $x_0\in E^c(\sigma(\cdot+n))$,
\[
\{x_0, S_2(x_0), S_2S_2(x_0)\}\subset\mathbb{K}_{\|\cdot\|_P}(S_2),\quad
 S_2S_2S_2(x_0)\in\mathbb{K}_{\|\cdot\|_P}(S_1).
\]
If $x_0$ and $S_2(x_0)$ are linear dependent, that is,
\[
 S_2(x_0)=\lambda x_0,
\]
for some $\lambda$ with $|\lambda|=1$, then
$S_2S_2S_2(x_0)=\lambda^3x_0\in\mathbb{K}_{\|\cdot\|_P}(S_2)$. So
\[
 S_2S_2S_2(x_0)\in \mathbb{K}_{\|\cdot\|_P}(S_1)\cap\mathbb{K}_{\|\cdot\|_P}(S_2)=\{0\},
\]
which implies that $x_0=0$. On the other hand ,if $x_0$ and $S_2(x_0)$ are linear independent, then
\[
 S_2S_2(x_0)=\lambda x_0+\alpha S_2(x_0),
\]
for some $\lambda$ and $\alpha$, since $\dim\mathbb{K}_{\|\cdot\|_P}(S_2)=2$. Thus
$S_2S_2S_2(x_0)$ is a linear combination of $S_2(x_0)$ and $S_2S_2(x_0)$. So it
is also in $\mathbb{K}_{\|\cdot\|_P}(S_2)$. Therefore
\[
S_2S_2S_2(x_0)\in\mathbb{K}_{\|\cdot\|_P}(S_1) \cap\mathbb{K}_{\|\cdot\|_P}(S_2)=\{0\},
\]
which shows $x_0=0$. Thus $E^c(\sigma(\cdot+n))=\{0\}$ and then $E^c(\sigma)=\{0\}$.

Similarly, when $\dim\mathbb{K}_{\|\cdot\|_P}(S_1)=2$ and $\dim\mathbb{K}_{\|\cdot\|_P}(S_2)=1$, we
can prove that all generic recurrent switching signals, but the following
four periodic switching signals
\[
(1,1,1,\dotsc),\ (2,2,2,\dotsc),\  (\widehat{2,1},\widehat{2,1},\dotsc),
\ (\widehat{2,1,1},\widehat{2,1,1},\dotsc),
\]
are stable for $\bS$.

This completes the proof of Theorem~\ref{thmIV.4}.
\end{proof}

We have the following remarks on Theorem~\ref{thmIV.4}.

\begin{remark}
Similarly, we can consider a switched linear system composed of two
subsystems on $\mathbb{R}^d$ with $d\geq 4$. In this case, under the
assumptions (\ref{eq1.3a}) and (\ref{eqIV.1}), if either
$\mathbb{K}_{\|\cdot\|_P}(S_1)$ or $\mathbb{K}_{\|\cdot\|_P}(S_2)$ has dimension $1$, then all
generic recurrent switching signals but finitely many periodic signals are
stable for $\bS$.
\end{remark}

\begin{remark}\label{remIV.2}
Under the assumptions on Theorem~\ref{thmIV.4}, in order to obtain
the stability for all recurrent switching signals, we just need to
check finitely many periodic signals to see whether they are stable for $\bS$.
\end{remark}

\begin{remark}
Theorem~\ref{thmIV.4} suggests a easy computable sufficient condition
of asymptotically stable for switched linear systems which are
composed of two subsystems. In fact, Remark~\ref{remIV.2} provides a
direct way to check the stability of all recurrent signals, which
implies the asymptotically stable of the systems by
Lemma~\ref{lemIV.1}.
\end{remark}

We can also discuss the stability of switched linear systems
composed of finite many subsystems similarly. But it is troublesome
to formulate the corresponding assumptions. Here we will give an
example to illustrate such conditions in Section~\ref{secVI}.

%%%%%%%%%%%%%%%%%%%%%%%%%%%%%%%%%%%%%%%
\subsection{Almost sure stability}\label{secIV.2}%%%
Let $(\varSigma_{\!K}^+, \mathscr{B})$ be the Borel $\sigma$-field of the space $\varSigma_{\!K}^+$ and then the one-sided Markov shift map $\theta\colon\sigma(\cdot)\mapsto\sigma(\cdot+1)$ is measurable.
A Borel probability measure $\mathbb{P}$ on $\varSigma_{\!K}^+$ is said to be
\emph{$\theta$-invariant}, if $\mathbb{P}=\mathbb{P}\circ \theta^{-1}$, i.e.
$\mathbb{P}(B)=\mathbb{P}(\theta^{-1}(B))$ for all $B\in\mathscr{B}$. A
$\theta$-invariant probability measure $\mathbb{P}$ is called
\emph{$\theta$-ergodic}, provided that for
$B\in\mathscr{B}$, $\mathbb{P}\left((B\setminus
\theta^{-1}(B))\cup(\theta^{-1}(B)\setminus B)\right)=0$ implies $\mathbb{P}(B)=1$ or $0$.

An ergodic measure $\mathbb{P}$ is called \emph{non-atomic}, if every singleton set $\{\sigma\}$ has $\mathbb{P}$-measure $0$.

Using Theorem~\ref{thmIV.4}, we can easily prove Theorem~B stated in Section~\ref{secI.4}.

\begin{proof}[Proof of Theorem~B]%%%
Let $\mathbb{P}$ be an arbitrary non-atomic $\theta$-ergodic measure on $\varSigma_{2}^+$. Then from the Poincar\'{e} recurrence theorem (see, e.g., \cite[Theorem~1.4]{Wal82}), it follows that $\mathbb{P}$-a.e. $\sigma\in\varSigma_{2}^+$ are recurrent. In addition, sine $\mathbb{P}$ is non-atomic, we obtain that $\mathbb{P}$-a.e. $\sigma\in\varSigma_{2}^+$ are non-periodic and generic. This completes the proof of Theorem~B from Theorem~\ref{thmIV.4}.
\end{proof}

We note that in the proof of Theorem~B presented above, the deduction of the genericity of $\sigma$ needs the assumption $K=2$.
%%%%%%%%%%%%%%%%%%%%%%%%%%%%%%%%%%%%%%%%%%%%%%%%%%%%%%%%%%%%%%%%%
%%%%%%%%%%%%%%%%%%%%%%%%%%%%%%%%%%%%%%%%%%%%%%%%%%%%%%%%%%%%%%%%%
\section[Absolute stability of a pair of matrices]{Absolute stability of a pair of matrices with a weak Lyapunov matrix}\label{secV}%%%
We now deal with the case $\bS=\{S_1,S_2\}\subset\mathbb{R}^{d\times d}$, where $S_1$ and $S_2$ both are stable and share
a common, but not necessarily strict, quadratic Lyapunov function. For any $A\in\mathbb{R}^{d\times d}$, we denote by $\rho(A)$ the spectral radius of $A$.

Our first absolute stability result Theorem~C is restated as follows:

\begin{theorem}\label{thm5.1}%%
Let $\bS=\{S_1,S_2\}\subset\mathbb{R}^{2\times 2}$ satisfy condition (\ref{eq1.3a}). Then, $\bS$ is absolutely stable (i.e., $\|S_{\sigma_{\!n}}\dotsc S_{\sigma_{\!1}}\|\to0$ as $n\to+\infty$, for all switching signals $\sigma\in\varSigma_{2}^+$) if and only if there holds that $\rho(S_1)<1, \rho(S_2)<1$, and $\rho(S_1S_2)<1$.
\end{theorem}

\begin{proof}
We only need to prove the sufficiency. Let $\rho(S_1)<1, \rho(S_2)<1$, and $\rho(S_1S_2)<1$. Let $\sigma=(\sigma_{\!n})_{n=1}^{+\infty}\in\varSigma_{2}^+$ be
an arbitrary recurrent switching signal. Clearly, if $\sigma$ is not generic, then it is stable for $\bS$. So we assume $\sigma$ is generic and recurrent. Then, from Theorem~\ref{thmIII.3} there exists a splitting of $\mathbb{R}^2$ into subspaces:
\begin{equation*}
\mathbb{R}^2=E^s(\sigma)\oplus E^c(\sigma).
\end{equation*}
If $\dim E^c(\sigma)=0$, then $\sigma$ is stable for $\bS$; and if $\dim E^c(\sigma)=2$ then either $\rho(S_1)=1$ or $\rho(S_2)=1$, a contradiction. We now assume $\dim E^c(\sigma)=1$.

Then, $\dim\mathbb{K}_{\|\cdot\|_P}(S_1)=1$ and $\dim\mathbb{K}_{\|\cdot\|_P}(S_2)=1$. It might be assumed, without loss of generality, that $\sigma_{\!1}=1$ and then we have $\mathbb{K}_{\|\cdot\|_P}(S_1)=E^c(\sigma)$. From this, we see
\begin{equation*}
\sigma_2=2,\; \sigma_3=1,\;\dotsc,\; \sigma_{2n}=2,\;  \sigma_{2n+1}=1,\;\dotsc.
\end{equation*}
This contradicts $\rho(S_1S_2)=\rho(S_2S_1)<1$.

Therefore, $E^c(\sigma)=\{0\}$ and $\bS$ is absolutely stable from Lemma~\ref{lemIV.1}.
\end{proof}

So, Theorem~C is proved.

Next, we need a simple fact for considering higher dimensional cases.

\begin{lem}[{\cite[Corollary]{VL}}]\label{lemV.2}%%%
Let $A\in\mathbb{R}^{d\times d}$ be a stable matrix (i.e., $\rho(A)<1$) such that
\begin{equation*}
D-A^TDA\ge0
\end{equation*}
for some symmetric, positive-definite matrix $D$. Then $D-(A^d)^TDA^d>0$.
\end{lem}

This lemma refines Lemma~\ref{lemIII.2}. From it, we can obtain a simple result which improves the statement of Theorem~A in the case of $d=2$ and $K=2$.

\begin{cor}\label{corV.3}%%%
Let $\bS=\{S_1,S_2\}\subset\mathbb{R}^{2\times 2}$ satisfy condition (\ref{eq1.3a}). If $\rho(S_1)<1$ and $\rho(S_2)<1$, then for any $\theta$-ergodic probability
measure $\mathbb{P}$ on $\varSigma_2^+$, $\bS$ is stable driven by $\mathbb{P}$-a.e. $\sigma\in\varSigma_2^+$ as long as $\mathbb{P}$ satisfies $\mathbb{P}(\{(12,12,12,\dotsc), (21,21,21,\dotsc)\})=0$.
\end{cor}

\begin{proof}
Since $\mathbb{P}$ is ergodic and $\mathbb{P}(\{(12,12,12,\dotsc), (21,21,21,\dotsc)\})=0$, we have
\begin{equation*}
\mathbb{P}(\{\sigma\in\varSigma_2^+\,|\,\sigma(\cdot+n)=(12,12,12,\dotsc)\textrm{ or }(21,21,21,\dotsc)\textrm{ for some }n\ge1\})=0.
\end{equation*}
Now, let $\sigma=(\sigma_{\!n})_{n=1}^{+\infty}\in\varSigma_2^+$ be arbitrary. Then, $\sigma$ can consist of the following $2$-length words:
\begin{equation*}
11,\;22,\;12,\;21.
\end{equation*}
If $11$ (or $22$) appears infinitely many times in $(\sigma_{\!n})_{n=1}^{+\infty}$, then from Lemma~\ref{lemV.2} it follows that $\bS$ is stable driven by $\sigma$. Next, assume $11$ and $22$ both only appear finitely many times in $(\sigma_{\!n})_{n=1}^{+\infty}$ and let $a=12$ and $b=21$. Then, one can find some $N\ge1$ such that
\begin{equation*}
\sigma(\cdot+N)=(a,a,a,\dotsc).
\end{equation*}
Note here that if $ab$ appears $m$ times in $(\sigma_{\!n})_{n=1}^{+\infty}$ then $22$ must appear $m$ times; if $ba$ appears $m$ times in $(\sigma_{\!n})_{n=1}^{+\infty}$ then $11$ must appear $m$ times.
So, $\bS$ is stable driven by $\mathbb{P}$-a.e. $\sigma\in\varSigma_2^+$.

This completes the proof of Corollary~\ref{corV.3}.
\end{proof}

The condition $\mathbb{P}(\{(12,12,12,\dotsc), (21,21,21,\dotsc)\})=0$ means that $\mathbb{P}$ is not distributed on the periodic orbit of the one-sided Markov shift $(\varSigma_{\!K}^+,\theta)$:
\begin{equation*}
\{(12,12,\dotsc),\; (21,21,\dotsc)\}.
\end{equation*}
This corollary shows that $\bS$ is ``completely'' almost sure stable up to only one ergodic measure supported on a periodic orbit generated by the word $12$.

In addition, Theorem~C can be directly deduced from Corollary~\ref{corV.3} and Lemma~\ref{lemIV.1}.

For the sake of our convenience, we now restate our second absolute stability result Theorem~D as follows:

\begin{theorem}\label{thmV.4}%%
Let $\bS=\{S_1,S_2\}\subset\mathbb{R}^{3\times 3}$ satisfy condition (\ref{eq1.3a}). Then, $\bS$ is absolutely stable if and only if there holds the following conditions:
\begin{equation*}
\rho(S_1)<1,\quad \rho(S_2)<1,\eqno{(\mathrm{C1})}
\end{equation*}
\begin{equation*}
\rho(S_1S_2)<1,\eqno{(\mathrm{C2})}
\end{equation*}
\begin{equation*}
\rho(S_{w_1}S_{w_2}S_{w_3})<1\quad \forall (w_1,w_2,w_3)\in\{1,2\}^3,\eqno{(\mathrm{C3})}
\end{equation*}
\begin{equation*}
\rho(S_{w_1}\dotsm S_{w_4})<1\quad \forall (w_1,\dotsc,w_4)\in\{1,2\}^4,\eqno{(\mathrm{C4})}
\end{equation*}
\begin{equation*}
\rho(S_{w_1}\dotsm S_{w_5})<1\quad \forall (w_1,\dotsc,w_5)\in\{1,2\}^5,\eqno{(\mathrm{C5})}
\end{equation*}
\begin{equation*}
\rho(S_{w_1}\dotsm S_{w_6})<1\quad \forall (w_1,\dotsc,w_6)\in\{1,2\}^6,\eqno{(\mathrm{C6})}
\end{equation*}
\begin{equation*}
\rho(S_{w_1}\dotsm S_{w_8})<1\quad \forall (w_1,\dotsc,w_8)\in\{1,2\}^8.\eqno{(\mathrm{C8})}
\end{equation*}
\end{theorem}

We note here that it is somewhat surprising that we do not need to consider the words of length $7$.

\begin{proof}
We need to consider only the sufficiency. Let conditions (C1) -- (C8) all hold. According to Lemma~\ref{lemIV.1}, we let $\sigma=(\sigma_{\!n})_{n=1}^{+\infty}\in\varSigma_{2}^+$ be an arbitrary recurrent switching signal. There is no loss of generality in assuming $\sigma_{\!1}=1$.

It is easily seen that $0\le\dim\mathbb{K}_{\|\cdot\|_P}(S_1)\le 2$ and $0\le\dim\mathbb{K}_{\|\cdot\|_P}(S_2)\le 2$ by condition (C1). Then from Theorem~\ref{thmIII.3} with $\|\cdot\|=\|\cdot\|_P$, there exists a splitting of $\mathbb{R}^3$ into subspaces:
\begin{equation*}
\mathbb{R}^3=E^s(\sigma)\oplus E^c(\sigma)\quad \textrm{such that }\dim E^c(\sigma)\le\dim\mathcal{K}_{k,\|\cdot\|_P}\textrm{ for }k=1,2.
\end{equation*}
There is only one of the following three cases occurs.
\begin{itemize}
\item $\dim E^c(\sigma)=2$;

\item $\dim E^c(\sigma)=1$;

\item $\dim E^c(\sigma)=0$.
\end{itemize}
Clearly, if $\sigma$ is not generic, then it is stable for $\bS$. So we let $\sigma$ be generic in what follows. We also note that $E^c(\sigma)\subseteq\mathbb{K}_{\|\cdot\|_P}(S_1)$.

Case (a): Let $\dim E^c(\sigma)=2$. Then $\dim\mathbb{K}_{\|\cdot\|_P}(S_1)=\dim\mathbb{K}_{\|\cdot\|_P}(S_2)=2$ and further we have
$\mathbb{K}_{\|\cdot\|_P}(S_1)=E^c(\sigma)$. If $\sigma_2=1$ then it follows that $\mathbb{K}_{\|\cdot\|_P}(S_1)$ is $S_1$-invariant and so $\rho(S_1)=1$ by Lemma~\ref{lemIV.2}, a contradiction. Thus, $\sigma_2=2$. If $\sigma_3=2$ it follows that $\mathbb{K}_{\|\cdot\|_P}(S_2)$ is $S_2$-invariant and so $\rho(S_2)=1$ by Lemma~\ref{lemIV.2}, also a contradiction. So, $\sigma_3=1$. Repeating this, we can see $\sigma=(1,2,1,2,1,2,\dotsc)$, a contradiction to condition (C2).
Thus, the case (a) cannot occur.

Case (b): Let $\dim E^c(\sigma)=1$. (This is the most complex case needed to discussion.) We first claim that $\sigma$ does not contain any one of the following two words:
\begin{equation*}
(1,1,1),\; (2,2,2).
\end{equation*}
In fact, without loss of generality, we let $(\sigma_{n+1},\sigma_{n+2},\sigma_{n+3})=(2,2,2)$. Choose a vector $x\in E^c(\sigma)$ with $\|x\|_P=1$. Then, $v:=S_{\sigma_{\!n}}\dotsm S_{\sigma_{\!1}}(x)\in\mathbb{K}_{\|\cdot\|_P}(S_2)$ with $\|v\|_P=1$. Moreover, $S_2(v)$ and $S_2(S_2(v))$ both belong to $\mathbb{K}_{\|\cdot\|_P}(S_2)$ such that with $\|S_2(v)\|_P=\|S_2(S_2(v))\|_P=1$. Since
$S_2(v)\not=\pm v$ (otherwise $\rho(S_2)=1$), we see $\mathbb{K}_{\|\cdot\|_P}(S_2)$ is $S_2$-invariant. So, $\rho(S_2)=1$ by Lemma~\ref{lemIV.2}, a contradiction to condition (C1).

Secondly, we claim that if $\sigma$ contains the word of the form $(1,1,w_1,\dotsc,w_m,1,1)$ then
\begin{equation*}
\rho(S_{w_m}\dotsc,S_{w_1}S_1S_1)=1;
\end{equation*}
and if $\sigma$ contains the word of the form $(2,2,w_1,\dotsc,w_m,2,2)$ then
\begin{equation*}
\rho(S_{w_m}\dotsc,S_{w_1}S_2S_2)=1.
\end{equation*}
In fact, without loss of generality, we assume that
\begin{equation*}
\sigma=(1,\sigma_2,\dotsc,\sigma_{\!n}, 2,2,w_1,\dotsc,w_{m},2,2,\dotsc).
\end{equation*}
Then, take arbitrarily a vector $x\in E^c(\sigma)$ with $\|x\|_P=1$ and write $v_n:=S_{\sigma_{\!n}}\dotsm S_{\sigma_{\!1}}(x)$. So, $v_n$ and $S_2(v_n)$ both belong to $\mathbb{K}_{\|\cdot\|_P}(S_2)$ such that $\|v_n\|_p=\|S_2(v_n)\|_P=1$. On the other hand, $v^\prime:=S_{w_m}\dotsm S_{w_1}S_2S_2(v_n)$ and $S_2(v^\prime)$ both belong to $\mathbb{K}_{\|\cdot\|_P}(S_2)$ with $\|v^\prime\|_P=\|S_2(v^\prime)\|_P=1$. If $v_n\not=\pm v^\prime$ then $\mathbb{K}_{\|\cdot\|_P}(S_2)$ is $S_2$-invariant and so $\rho(S_2)=1$ by Lemma~\ref{lemIV.2}, a contradiction to condition (C1). Thus, we have $v_n=\pm v^\prime$ and then $\rho(S_{w_m}\dotsc,S_{w_1}S_2S_2)=1$.

Thirdly, we show the case (b), i.e., $\dim E^c(\sigma)=1$, does not occur too. In fact, from the above claims, it follows that $\sigma=(\sigma_{\!n})_{n=1}^{+\infty}$ only possesses the following forms:
\begin{equation}\label{eqV.1}
1\to\begin{cases}12\to\dotsm~\textrm{(case (A))}\\2\to\begin{cases}1\to\dotsm~\textrm{(case (B))}\\21\to\dotsm~\textrm{(case (C))}\end{cases}\end{cases}
\end{equation}
Here and in the sequel, ``$a\to b$" means that $b$ follows $a$; i.e., $\sigma_{\!n}=a$ and $\sigma_{\!{n+1}}=b$ for some $n$. For example, in the above figure, ``$1\to2\to21$" means $\sigma_{\!1}=1, \sigma_{\!2}=2$ and $(\sigma_{\!3},\sigma_{\!4})=(2,1)$. In addition, in the following three figures, the symbol ``$\times$" means ``This case does not happen." For the case (A) in the above figure (\ref{eqV.1}), we have the following:
\begin{equation*}
112\to\begin{cases}
1\to\begin{cases}
1~(\times\textrm{ by (C3)})\\2\to\begin{cases}1\to\begin{cases}1~(\times\textrm{ by (C5)})\\2~(\times\textrm{ by (C2)}\textrm{ and Lemma}~\ref{lemV.2})\end{cases}\\
21\to\begin{cases}1~(\times\textrm{ by (C6)})\\2\to\begin{cases}1\to\begin{cases}1~(\times\textrm{ by (C8)})\\2\to\begin{cases}1~(\times\textrm{ by (C2) and Lemma~\ref{lemV.2}})\\2~(\times\textrm{ by (C5)})\end{cases}\end{cases}\\2~(\times\textrm{ by (C3)})\end{cases}\end{cases}
\end{cases}\end{cases}\\21\to\begin{cases}1~(\times\textrm{ by (C4)})\\2\to\begin{cases}1\to\begin{cases}1~(\times\textrm{ by (C6)})\\2\to\begin{cases}1~(\times\textrm{ by (C2) and Lemma~\ref{lemV.2}})\\2~(\times\textrm{ by (C5)})\end{cases}\end{cases}\\2~(\times\textrm{ by (C3)})\end{cases}\end{cases}\end{cases}
\end{equation*}
Thus,
\begin{equation*}
(\sigma_{\!1},\sigma_{\!2},\sigma_{\!3})\not=(1,1,2)
\end{equation*}
and then
\begin{equation}\label{eqV.2}
(\sigma_{\!1},\sigma_{\!2})\not=(1,1).
\end{equation}

For the case (C) in the figure (\ref{eqV.1}), we have
\begin{equation*}
1221\to\begin{cases}12\to\begin{cases}1\to\begin{cases}1~(\times\textrm{ by (C3)})\\2\to\begin{cases}1\to\begin{cases}1~(\times\textrm{ by (C5)})\\2~(\times\textrm{ by (C6)})\end{cases}\\2~(\times\textrm{ by (C6)})\end{cases}\end{cases}\\2~(\times\textrm{ by (C4)})\end{cases}\\
2\to
\begin{cases}
1\to\begin{cases}12\to\begin{cases}1\to\begin{cases}1~(\times\textrm{ by (C3)})\\2\to\begin{cases}1\to\begin{cases}1~(\times\textrm{ by (C5)})\\2~(\times\textrm{ Lemma~\ref{lemV.2}})\end{cases}\\2~(\times\textrm{ by (C8)})\end{cases}\end{cases}\\2~(\times\textrm{ by (C6)})\end{cases}\\2\to\begin{cases}1~(\times\textrm{ by Lemma~\ref{lemV.2}})\\2~(\times\textrm{ by (C5)})\end{cases}\end{cases}\\2~(\times\textrm{ by (C3)}) \end{cases}\end{cases}
\end{equation*}
Thus
\begin{equation*}
(\sigma_{\!1},\sigma_{\!2},\sigma_{\!3},\sigma_{\!4})\not=(1,2,2,1)
\end{equation*}
and then
\begin{equation}\label{eqV.3}
(\sigma_{\!1},\sigma_{\!2},\sigma_{\!3})\not=(1,2,2).
\end{equation}

Finally, for the case (B) in the figure (\ref{eqV.1}),
\begin{equation*}
121\to
\begin{cases}
12\to\begin{cases}1\to\begin{cases}1~(\times\textrm{ by (C3)})\\2\to\begin{cases}1\to\begin{cases}1~(\times\textrm{ by (C5)})\\2~(\times\textrm{ by Lemma~\ref{lemV.2}})\end{cases}\\
21\to\begin{cases}1~(\times\textrm{ by (C6)})\\2\to\begin{cases}1\to\begin{cases}1~(\times\textrm{ by (C8)})\\2\to\begin{cases}1~(\times\textrm{ by Lemma~\ref{lemV.2}})\\2~(\times\textrm{ by (C5)})\end{cases}\end{cases}\\2~(\times\textrm{ by (C3)})\end{cases}\end{cases}\end{cases}\end{cases}\\
21\to\begin{cases}1~(\times\textrm{ by (C4)})\\2\to\begin{cases}1\to\begin{cases}1~(\times\textrm{ by (C6)})\\2\to\begin{cases}1~(\times\textrm{ by Lemma~\ref{lemV.2}})\\2~(\times\textrm{ by (C5)})\end{cases}\end{cases}\\2~(\times\textrm{ by (C3)})\end{cases}\end{cases}\end{cases}\\
2\to\begin{cases}1\to\begin{cases}12\to\begin{cases}1\to\begin{cases}1~(\times\textrm{ by (C3)})\\2\to\begin{cases}1\to\begin{cases}1~(\times\textrm{ by (C5)})\\2~(\times\textrm{ by Lemma~\ref{lemV.2}})\end{cases}\\
21\to\begin{cases}1~(\times\textrm{ by (C6)})\\2\to\begin{cases}1\to\begin{cases}1~(\times\textrm{ by (C8)})\\2\to\begin{cases}1~(\times\textrm{ by Lemma~\ref{lemV.2}})\\2~(\times\textrm{ by (C5)})\end{cases}\end{cases}\\2~(\times\textrm{ by (C3)})\end{cases}\end{cases}\end{cases}\end{cases}\\
21\to\begin{cases}1~(\times\textrm{ by (C4)})\\2\to\begin{cases}1\to\begin{cases}1~(\times\textrm{ by (C6)})\\2\to\begin{cases}1~(\times\textrm{ by Lemma~\ref{lemV.2}})\\2~(\times\textrm{ by (C5)})\end{cases}\end{cases}\\2~(\times\textrm{ by (C3)})\end{cases}\end{cases}\end{cases}\\2~(\times\textrm{ by Lemma~\ref{lemV.2}})\end{cases}\\
21\to\begin{cases}12\to\begin{cases}1\to\begin{cases}1~(\times\textrm{ by (C3)})\\21\to\begin{cases}1~(\times\textrm{ by (C5)})\\2~(\times \textrm{ by Lemma~\ref{lemV.2}})\end{cases}\end{cases}\\2~(\times\textrm{ by (C4)})\end{cases}\\
2\to\begin{cases}1\to\begin{cases}12\to\begin{cases}1\to\begin{cases}1~(\times\textrm{ by (C3)})\\2\to\begin{cases}1\to\begin{cases}1~(\times\textrm{ by (C5)})\\2~(\times\textrm{ by Lemma~\ref{lemV.2}})\end{cases}\\2~(\times\textrm{ by (C8)})\end{cases}\end{cases}\\2~(\times\textrm{ by (C6)})\end{cases}\\2\to\begin{cases}1~(\times\textrm{ by Lemma~\ref{lemV.2}})\\2~(\times\textrm{ by (C5)})\end{cases}\end{cases}\\2~(\times\textrm{ by (C3)})\end{cases}
\end{cases}\end{cases}\end{cases}
\end{equation*}
Thus, $(\sigma_{\!1},\sigma_{\!2},\sigma_{\!3})\not=(1,2,1)$. Further, from (\ref{eqV.3}) it follows $(\sigma_{\!1},\sigma_{\!2})\not=(1,2)$. This together with (\ref{eqV.2}) implies that $(\sigma_{\!1},\sigma_{\!2})\not\in\{(1,1),(1,2)\}$, a contradiction.

So, $\dim E^c(\sigma)\not=1$ and hence case (b) does not occur.

Therefore, $\dim E^c(\sigma)=0$. This implies that $\sigma$ is stable for $\bS$. Therefore $\bS$ is absolutely stable from Lemma~\ref{lemIV.1}.

This completes the proof of Theorem~\ref{thmV.4}.
\end{proof}

%%%%%%%%%%%%%%%%%%%%%%%%%%%%%%%%%%%%%%%%%%%%%%%%%%%%
%%%%%%%%%%%%%%%%%%%%%%%%%%%%%%%%%%%%%%%%%%%%%%%%%%%%
\section{Examples}\label{secVI}%%%%

We in this section shall give several examples to illustrate
applications of our results. In what follows, let $\|\cdot\|_2$ be the usual Euclidean norm on $\mathbb{R}^d$; that is, $P=I_d$ in (\ref{eq1.3b}).

First, a very simple example is the following.

\begin{example}%%%%
Let $\bS=\{S_1,\ S_2\}$ with
\[
  S_1=\left(\begin{array}{cc}
       1 & 0\\ 0 & \alpha
  \end{array}\right),
  \qquad S_2=\left(\begin{array}{cc}
       \beta & 0\\ 0 & 1
  \end{array}\right),
\]
where $|\alpha|<1,\ \ |\beta|<1$. It is easy to see that
\[
  \|S_1\|_2=\|S_2\|_2=1,
\]
and that $\mathbb{K}_{\|\cdot\|_2}(S_1)=\{(x_1, 0)^T\in\mathbb{R}^2\mid
x_1\in\mathbb{R}\},\ \mathbb{K}_{2,\|\cdot\|_2}(S_2)=\{(0, x_2)^T\in\mathbb{R}^2\mid
x_2\in\mathbb{R}\}$. So, we can obtain that $\mathbb{K}_{\|\cdot\|_2}(S_1)\bigcap\mathbb{K}_{\|\cdot\|_2}(S_2)=\{0\}$
and $\mathbb{K}_{\|\cdot\|_2}(S_k)$ is $S_k$-invariant. Thus the switched linear
system $\bS$ is asymptotically stable for all switching signals in
which each $k$ in $\{1,2\}$ is stable by Lemma~\ref{lemIV.3}. Also,
from Theorem~\ref{thmIV.4}, it follows that all recurrent signals but
the fixed signals $(1,1,1,\dotsc)$ and $(2,2,2,\dotsc)$ are
stable for $\bS$. We note here that the periodic switching signal $(1,2,1,2,\dotsc)$ is stable for $\bS$.
\end{example}

A more interesting example is the following

\begin{example}\label{exaVI.2}
Let $\bS=\{S_1,\ S_2\}$ with
\[
  S_1=\alpha\left(\begin{array}{cc}
       1 & 0\\ 1 & 1
  \end{array}\right),
  \quad S_2=\beta\left(\begin{array}{cc}
       1 & \frac{3}{2}\\ 0 & 1
  \end{array}\right),
\]
where
\[
 \alpha=\sqrt{\frac{3-\sqrt{5}}{2}},\quad \beta=\frac{1}{2}.
\]
Then, $\|S_1\|_2=\|S_2\|_2=1$. A direct computation shows that
\[
\mathbb{K}_{\|\cdot\|_2}(S_1)=\left\{(x_1,x_2)^T\in\mathbb{R}^2\mid
x_1=\frac{\sqrt{5}+1}{2}x_2\right\}
\]
\[
\mathbb{K}_{\|\cdot\|_2}(S_2)=\left\{(x_1,x_2)^T\in\mathbb{R}^2\mid x_2=2x_1\right\}.
\]
Thus $\mathbb{K}_{\|\cdot\|_2}(S_1)\bigcap\mathbb{K}_{\|\cdot\|_2}(S_2)=\{0\}$. But they are not
invariant. Thus $\bS$ is asymptotically stable for all generic recurrent
switching signals but the periodic signal $(1,2,1,2,\dotsc)$ by
Theorem~\ref{thmIV.4}. Note that the two subsystems themselves are
asymptotically stable.
\end{example}

Next, we give an example which is the discretization of the
switched linear continuous system borrowed from \cite{Bal}.

\begin{example}
Let $\bS=\{S_1,\ S_2, S_3\}$ with
\begin{equation*}
  S_1=\left(\begin{array}{ccc}
       \alpha & 0 & 0\\ 0 & 0 & -1\\0 & 1 & 0
  \end{array}\right),
  \qquad S_2=\left(\begin{array}{ccc}
       \alpha & 0 & 0\\ 0 & \alpha & 0\\0 & 0 & 1
  \end{array}\right), \qquad
  S_3=\left(\begin{array}{ccc}
       \alpha & 0 & 0\\ 0 & 1 & 0\\0 & 0 & \alpha
  \end{array}\right),
\end{equation*}
where $|\alpha|<1$.
It is easy to see that $\|S_1\|_2=\|S_2\|_2=\|S_3\|_2=1$ and
\begin{align*}
\mathbb{K}_{\|\cdot\|_2}(S_1)& =\left\{(x_1,x_2,x_3)^T\in\mathbb{R}^3\mid x_1=0\right\},\\
\mathbb{K}_{\|\cdot\|_2}(S_2)& =\left\{(x_1,x_2,x_3)^T\in\mathbb{R}^3\mid x_1=x_2=0\right\},\\
\mathbb{K}_{\|\cdot\|_2}(S_3)& =\left\{(x_1,x_2,x_3)^T\in\mathbb{R}^3\mid
x_1=x_3=0\right\}.
\end{align*}

Since $\mathbb{K}_{\|\cdot\|_2}(S_2)\bigcap\mathbb{K}_{\|\cdot\|_2}(S_3)=\{0\}$ and they are
invariant respect to $S_2$ and $S_3$, respectively, we have that any
generic switching signal in which either the word $(2,3)$ or the
$(3,2)$ appears infinitely many times are stable by
Lemma~\ref{lemIV.3}. For the any other generic switching signals
$\sigma=(\sigma_{\!1},\sigma_{\!2},\dotsc)$, that is, in which both the word $(2,3)$ and
the $(3,2)$ appear at most finite many times, the matrix $Q_{\sigma}$
defined in (\ref{eqIII.1}) is
\[
Q_{\sigma}=\left(\begin{array}{ccc}
       0 & 0 & 0\\ 0 & \alpha^{k_0} & 0\\0 & 0 & \alpha^{j_0}
  \end{array}\right).
\]
for some nonnegative integers $k_0$ and $j_0$ which depend on the
times of appearance of $(2,3)$ and $(3,2)$ in $\sigma$. Thus by
Proposition~\ref{propIII.7}, we have
\begin{align*}
&\lim_{n\to\infty}\|S_{\sigma_{\!n}}\dotsm S_{\sigma_{\!1}}x\|_2=0,&&\forall
   x\in \{(x_1,0, 0)^T\mid x_1\in\mathbb{R}\}=\ker(Q_{\sigma}),\\
&\lim_{n\to\infty}\|S_{\sigma_{\!n}}\dotsm S_{\sigma_{\!1}}x\|_2=\|Q_{\sigma}(x)\|_2,&&\forall
   x\in \{(0,x_2, x_3)^T\mid x_2,\
   x_3\in\mathbb{R}\}=\mathrm{Im}(Q_{\sigma}),
\end{align*}
for such kind of generic switching signals.
\end{example}

The following Example~\ref{exaVI.4} is associated to Theorem~C.

\begin{example}\label{exaVI.4}%%%
Let $\bS=\{S_1, S_2\}$ with
\begin{equation*}
  S_1=\frac{1}{2}\left(\begin{matrix}1&0\\\frac{3}{2}&-1\end{matrix}\right),
  \quad S_2=\sqrt{\frac{3-\sqrt{5}}{2}}\left(\begin{matrix}1&1\\0&1\end{matrix}\right).
\end{equation*}
Then, using $\sqrt{\rho(A^TA)}=\|A\|_2$ we have
\begin{equation*}
\rho(S_1)=\frac{1}{2}<1,\quad\|S_1\|_2=1\quad \textrm{and}\quad\rho(S_2)=\sqrt{\frac{3-\sqrt{5}}{2}}<1,\quad \|S_2\|_2=1.
\end{equation*}
In addition,
\begin{equation*}
\rho(S_1S_2)=\sqrt{\frac{3-\sqrt{5}}{2}}=\rho(S_2)<1.
\end{equation*}
Therefore, $\bS$ is absolutely stable by Theorem~C.
\end{example}

The interesting \cite[Proposition~18]{PW-LAA08} implies that if $\mathcal{S}=\{A_1,\dotsc,A_m\}\subset\mathbb{R}^{d\times d}$ is symmetric (i.e. $A^T\in\mathcal{S}$ whenever $A\in\mathcal{S}$), then $\mathcal{S}$ has the spectral finiteness property; in fact, it holds that $\rho(\mathcal{S})=\sqrt{\rho(A^TA)}$ for some $A\in\mathcal{S}$. This naturally motivates us to extend an arbitrary $\bS$ into a symmetric set $\mathcal{S}=\bS\cup\bS^T$. Let us see a simple example.

\begin{example}\label{exaVI.5}%%%
Let $\bS=\left\{A=\sqrt{\frac{3-\sqrt{5}}{2}}\left(\begin{matrix}1&1\\0&1\end{matrix}\right)\right\}$. Then, $\bS$ satisfies (\ref{eq1.3a}) with $\|A\|_2=1$ such that $\rho(\bS)=\sqrt{\frac{3-\sqrt{5}}{2}}<1$. But for $\mathcal{S}=\{A,A^T\}$, $\rho(\mathcal{S})=\sqrt{\rho(A^TA)}=1\not=\rho(\bS)$.
\end{example}

This example shows that the extension $\mathcal{S}$ does not work for the original system $\bS$ needed to be considered here.

Finally, the following Example~\ref{exaVI.6} is simple. Yet it is very interesting to the stability analysis of switched systems.

\begin{example}\label{exaVI.6}%%%
Let $\bS=\{S_1, S_2\}\subset\mathbb{R}^{2\times 2}$ with
\begin{equation*}
  S_1=\sqrt{\frac{3-\sqrt{5}}{2}}\left(\begin{matrix}1&0\\1&1\end{matrix}\right),
  \quad S_2=\sqrt{\frac{3-\sqrt{5}}{2}}\left(\begin{matrix}1&1\\0&1\end{matrix}\right).
\end{equation*}
Then, using $\sqrt{\rho(A^TA)}=\|A\|_2$ we have
\begin{equation*}
\rho(S_1)=\rho(S_2)=\sqrt{\frac{3-\sqrt{5}}{2}}<1,\quad\|S_1\|_2=\|S_2\|_2=1,\quad \textrm{and}\quad\rho(S_1S_2)=1.
\end{equation*}
So, $\bS$ is not absolutely stable. Yet from Corollary~\ref{corV.3}, $\bS$ is stable driven by $\mathbb{P}$-a.e. $\sigma\in\varSigma_2^+$, for any $\theta$-ergodic probability
measure $\mathbb{P}$ on $\varSigma_2^+$, as long as $\mathbb{P}$ is not the ergodic measure distributed on the periodic orbit
\begin{equation*}
\{(12,12,12,\dotsc), (21,21,21,\dotsc)\}.
\end{equation*}
\end{example}
%%%%%%%%%%%%%%%%%%%%%%%%%%%%%%%%%%%%%%%%%%%%%%%
%%%%%%%%%%%%%%%%%%%%%%%%%%%%%%%%%%%%%%%%%%%%%%%
\section{Concluding remarks}\label{secVII}%%%
In this paper, we have considered the asymptotic stability of a discrete-time linear switched system, which is induced by a set $\bS=\{S_1,\dotsc,S_K\}\subset\mathbb{R}^{d\times d}$ such that each $S_k$ shares a common, but not necessarily strict, Lyapunov matrix $P$ as in (\ref{eq1.3a}).

We have shown that if every subsystem $S_k$ is stable then $\bS$ is stable driven by a nonchaotic switching signal. Particularly, in the cases $K=2$ and $d=2,3$, we have proven that $\bS$ has the spectral finiteness property and so the stability is decidable.

Recall that $\bS$ is called \emph{periodically switched stable}, if
$\rho(S_{w_n}\dotsm S_{w_1})<1$ for all finite-length words $(w_1,\dotsc,w_n)\in\{1,\dotsc,K\}^n$ for $n\ge1$; see, e.g., \cite{Gur-LAA95, DHX-aut, Dai-JDE}.

Finally, we end this paper with a problem for further study.

\begin{conjecture}%%%
Let $\bS=\{S_1,S_2\}\subset\mathbb{R}^{d\times d}, d\ge4$, be an arbitrary pair such that condition (\ref{eq1.3a}). If $\bS$ is periodically switched stable, then it is absolutely stable. Equivalently, if $\rho(\bS)=1$ there exists at least one word $(w_1,\dotsc,w_n)\in\{1,2\}^n$ for some $n\ge1$ such that $\sqrt[n]{\rho(S_{w_n}\dotsm S_{w_1})}=1$.
\end{conjecture}

Since there exist uncountable many pairs $(\alpha,\gamma)\in(0,1)\times (0,1)$, for which
\begin{equation*}
\bS_{\alpha,\gamma}=\left\{S_1=\alpha\left(\begin{matrix}1&1\\0&1\end{matrix}\right),\quad S_2=\gamma\left(\begin{matrix}1&0\\1&1\end{matrix}\right)\right\}
\end{equation*}
is periodically switched stable such that $\|S_1\|=\|S_2\|=1$ under some extremal norm $\|\cdot\|$ on $\mathbb{R}^{2}$; but $\bS_{\alpha,\gamma}$ is not absolutely stable with $\rho(\bS_{\alpha,\gamma})=1$. See, for example, \cite{BM,BTV,Koz07,HMST}. So, condition (\ref{eq1.3a}) is very important for our Theorems B, C and D and for the above conjecture. In fact, the essential good of $\|\cdot\|_P$ is to guarantee that $\mathbb{K}_{\|\cdot\|_P}(S_1)$ and $\mathbb{K}_{\|\cdot\|_P}(S_2)$ are linear subspaces of $\mathbb{R}^d$ in our arguments.

%%%%%%%%%%%%%%%%%%%%%%%%%%%%%%%%%%%%%%%%%%%%%%%%%%%%%%%%%%%%%%%%%%%%%%%%%%
%%%%%%%%%%%%%%%%%%%%%%%%%%%%%%%%%%%%%%%%%%%%%%%%%%%%%%%%%%%%%%%%%%%%%%%%%%


\begin{thebibliography}{99}
\bibitem{BL00}
   \newblock {\sc A.~Bacciotti} and {\sc R.~Lionel},
   \newblock {\it Regularity of Liapunov functions for stable systems},
   \newblock {System \& Control Letters, 41 (2000), 265--270}.

\bibitem{BL05}
   \newblock {\sc A.~Bacciotti} and {\sc R.~Lionel},
   \newblock {\it Liapunov Functions and Stability in Control Theory}, 2nd ed.,
   \newblock {Comm. Control Engrg. Ser., Springer-Verlag, Berlin 2005}.

\bibitem{Bal}
   \newblock {\sc M.~Balde} and {\sc P.~Jouan},
   \newblock {\it Geometry of the limit sets of linear switched systems},
   \newblock {SIAM J. Control Optim., 49 (2011), 1048--1063}.

\bibitem{Bar}
   \newblock {\sc N.~Barabanov},
   \newblock {\it Lyapunov indicators of discrete inclusions I--III},
   \newblock {Autom. Remote Control, 49 (1988), 152--157, 283--287, 558--565}.

\bibitem{Bar88-SMJ}
   \newblock {\sc N.~Barabanov},
   \newblock {\it Absolute characteristic exponent of a class of linear non-stationary systems of differential equations},
   \newblock {Siberian Math. J., 29 (1988), 521--530}.

\bibitem{Bar93}
   \newblock {\sc N.~Barabanov},
   \newblock {\it On the Aizerman problem for 3rd-order nonstationary systems},
   \newblock {Differ. Equ., 29 (1993), 1439--448}.

\bibitem{BTV}
  \newblock {\sc V.\,D.~Blondel, J.~Theys} and {\sc A.\,A.~Vladimirov},
  \newblock {\it An elementary counterexample to the finiteness conjecture},
  \newblock {SIAM J. Matrix Anal. Appl., {24} (2003), 963--970}.

\bibitem{BM}
  \newblock {\sc T.~Bousch} and {\sc J.~Mairesse},
  \newblock {\it Asymptotic height optimization for topical IFS, Tetris heaps, and the finiteness conjecture},
  \newblock {J. Amer. Math. Soc., {15} (2002), 77--111}.

\bibitem{Dai-JMAA}
    \newblock {\sc X.~Dai},
    \newblock {\it Extremal and Barabanov semi-norms of a semigroup generated by a bounded family of matrices},
    \newblock {J. Math. Anal. Appl., 379 (2011), 827--833}.

\bibitem{Dai-JDE}
  \newblock {\sc X.~Dai},
  \newblock {\it Weakly recurrent switching signals, almost sure and partial stability of linear switched systems},
  \newblock {J. Differential Equations, {250} (2011), 3584-3629}.

\bibitem{DHX-ERA}
  \newblock {\sc X.~Dai, Y.~Huang} and {\sc M.~Xiao},
  \newblock {\it Realization of joint spectral radius via ergodic theory},
  \newblock {Electron. Res. Announc. Math. Sci., 18 (2011), 22-30}.

\bibitem{DHX-aut}
       \newblock {\sc X.~Dai, Y.~Huang} and {\sc M.~Xiao},
       \newblock {\it Periodically switched stability induces exponential stability of discrete-time linear switched systems in the sense of Markovian probabilities},
       \newblock {Automatica, 47 (2011), 1512--1519}.

\bibitem{DHX-AIM}
  \newblock {\sc X.~Dai, Y.~Huang} and {\sc M.~Xiao},
  \newblock {\it Pointwise stabilization of discrete-time matrix-valued stationary Markov chains},
  \newblock {Preprint 2011, arXiv:1107.0132v1 [math.PR]}.

\bibitem{Dai-Koz}
   \newblock {\sc X.~Dai} and {\sc V.~Kozyakin},
   \newblock {\it Finiteness property of a bounded set of matrices with uniformly sub-peripheral spectrum},
   \newblock {Information processes, 11 (2011), 253--261; arXiv:1106.2298v2 [math.FA]}.

\bibitem{DL-LAA92}
  \newblock {\sc I.~Daubechies} and {\sc J.\,C.~Lagarias},
  \newblock {\it Sets of matrices all infinite products of which converge},
  \newblock {Linear Algebra Appl., {161} (1992), 227--263. Corrigendum/addendum, {327} (2001), 69--83}.

\bibitem{Gur-LAA95}
  \newblock {\sc L.~Gurvits},
  \newblock {\it Stability of discrete linear inclusions},
  \newblock {Linear Algebra Appl., {231} (1995), 47--85}.

\bibitem{HMST}
  \newblock {\sc K.\,G.~Hare, I.\,D.~Morris, N.~Sidorov} and {\sc J.~Theys},
  \newblock {\it An explicit counterexample to the Lagarias-Wang finiteness conjecture},
  \newblock {Adv. Math., 226 (2011), 4667--4701}.

\bibitem{Hart02}
   \newblock {\sc D.\,J.~Hartfiel},
   \newblock {\it Nonhomogeneous Matrix Products},
   \newblock {World Scientific, New Jersey London, 2002}.

\bibitem{Koz07}
   \newblock {\sc V.\,S.~Kozyakin},
   \newblock {\it Structure of extremal trajectories of discrete linear systems and the finiteness conjecture},
   \newblock {Autom. Remote Control, {68} (2007), 174--209}.

\bibitem{MS}
   \newblock {\sc S.~Mendenhall} and {\sc G.\,L.~Slater},
   \newblock {\it A model for helicopter guidance on spiral trajectories},
   \newblock {in AIAA Guid. Control Conf., 1980, 62--71}.

\bibitem{NS60}
   \newblock {\sc V.\,V.~Nemytskii} and {\sc V.\,V.~Stepanov},
   \newblock {\it Qualitative Theory of Differential Equations},
   \newblock {Princeton University Press, Princeton, New Jersey 1960}.

\bibitem{PW-LAA08}{\sc E.~Plischke} and {\sc F.~Wirth},
  \newblock {\it Duality results for the joint spectral radius and transient behavior},
  \newblock {Linear Algebra Appl., {428} (2008), 2368--2384}.

\bibitem{RSD:aut}
  \newblock {\sc P.~Riedinger, M.~Sigalotti} and {\sc J.~Daafouz},
  \newblock {\it On the algebraic characterization of invariant sets of switched linear systems},
  \newblock {Automatica, {46} (2010), 1047--1052}.

\bibitem{SVR:IEEE}
  \newblock {\sc U.~Serres, J.-C.~Vivalda} and {\sc P.~Riedinger},
  \newblock {\it On the convergence of linear switched systems},
  \newblock {IEEE Trans. Automat. Control, {56} (2011), 320--332}.

\bibitem{Sun08}
    \newblock {\sc Z.~Sun},
    \newblock {\it A note on marginal stability of switched systems},
    \newblock {IEEE Trans. Automat. Control, 53 (2008), 625--631}.

\bibitem{VL}
  \newblock {\sc P.\,P.~Vaidyanathan} and {\sc V.~Liu},
  \newblock {\it An improved sufficient condition for absence of limit cycles in digital filters},
  \newblock {IEEE Trans. Circuits Systems, VOL. CAS-{34} (1987), 319--322}.

\bibitem{Wal82}
   \newblock {\sc P.~Walters},
   \newblock {\it An Introduction to Ergodic Theory},
   \newblock {GTM {79}, Springer-Verlag, New York, 1982}.

\bibitem{Wir02}
  \newblock {\sc F.~Wirth},
  \newblock {\it The generalized spectral radius and extremal norms},
  \newblock {Linear Algebra Appl., 342 (2002), 17--40}.

\bibitem{Zhou92}
  \newblock {\sc Z.~Zhou},
  \newblock {\it Weakly almost periodic point and measure center},
  \newblock {Sci. China Ser. A: Math., 36 (1992), 3019--3024}.
\end{thebibliography}
\end{document}